\documentclass[10pt, reqno]{amsart}
\usepackage{graphicx, amssymb, amsmath, amsthm}
\numberwithin{equation}{section}

\usepackage{amscd}
\def\pa{\partial}

\newcommand{\R}{\mathbb{R}}
\newcommand{\C}{\mathbb{C}}

\newcommand{\lk}{\mathcal{L}_K}

\newtheorem{theorem}{Theorem}[section]

\newtheorem{lemma}[theorem]{Lemma}

\newtheorem{corollary}[theorem]{Corollary}

\newtheorem{proposition}[theorem]{Proposition}

\theoremstyle{definition}

\newtheorem{remark}[theorem]{Remark}

\makeatletter
\newcommand{\Extend}[5]{\ext@arrow0099{\arrowfill@#1#2#3}{#4}{#5}}
\makeatother

\begin{document}
\title[NLS with Coulomb potential ]{Nonlinear Schr\"odinger equation with Coulomb potential}

\author{Changxing Miao}
\address{Institute of Applied Physics and Computational Mathematics, Beijing 100088}
\email{miao\_changxing@iapcm.ac.cn}

\author{Junyong Zhang}
\address{Department of Mathematics, Beijing Institute of Technology, Beijing 100081; Department of Mathematics, Cardiff University, UK}
\email{zhang\_junyong@bit.edu.cn; ZhangJ107@cardiff.ac.uk}

\author{Jiqiang Zheng}
\address{Institute of Applied Physics and Computational Mathematics, Beijing 100088}
\email{zhengjiqiang@gmail.com}

\begin{abstract}
In this paper, we study the Cauchy problem for the nonlinear
Schr\"odinger equations with Coulomb potential $i\pa_tu+\Delta
u+\tfrac{K}{|x|}u=\lambda|u|^{p-1}u$ with $1<p\leq5$ on $\R^3$.  We mainly consider the influence of the long range potential $K|x|^{-1}$ on the existence theory and scattering theory for nonlinear Schr\"odinger equation. In particular, we prove
the global existence when the Coulomb potential is attractive, i.e. $K>0$ and scattering theory when the Coulomb potential is repulsive i.e. $K\leq0$.
The argument is based on the interaction Morawetz-type inequalities and the equivalence of Sobolev norms.

\end{abstract}

 \maketitle

\begin{center}
 \begin{minipage}{100mm}
   { \small {{\bf Key Words:}  Nonlinear Schr\"odinger equation; global well-posedness; blow-up; scattering.}
      {}
   }\\
    { \small {\bf AMS Classification:}
      {35P25,  35Q55, 47J35.}
      }
 \end{minipage}
 \end{center}




\section{Introduction}

\noindent

\noindent We study the initial-value problem for the
nonlinear
 Schr\"odinger equations with Coulomb potential
\begin{align} \label{equ1.1}
\begin{cases}    (i\partial_t-\mathcal{L}_K)u= \lambda f(|u|^2)u,\quad
(t,x)\in\R\times\R^3,
\\
u(0,x)=u_0(x)\in H^1(\R^3), \quad x\in\R^3,
\end{cases}
\end{align}
where $u:\R_t\times\R_x^3\to \C,\; \mathcal{L}_K=-\Delta-\frac{K}{|x|}$ with $K\in\R$, $f(|u|^2)=|u|^{p-1}$, and $\lambda\in\{\pm1\}$
with $\lambda=1$ known as the defocusing case and $\lambda=-1$
as the focusing case.

The study of the operator $\mathcal{L}_K=-\Delta-K|x|^{-1}$ with the Coulomb potential originates from both the physical and mathematical interests.
In particular, $K$ is positive,  this operator provides a quantum mechanical description of the Coulomb force between
two charged particles and corresponds to having an external attractive long-range potential due to the presence of a positively charged atomic nucleus.
We refer to the reader to \cite{Mess, Ser}  for work on these more models of the hydrogen atom in quantum physics fields.

The mathematical interest in these equations however comes from the operator theory with a long range decay potential and the dispersive behavior of the solution. Note that $|x|^{-1}\in L^2(\R^3)+L^\infty(\R^3),$ we know from \cite[Theorem X.15]{RS} that $\mathcal{L}_K$ is essentially self-adjoint on $C_0^\infty(\R^3)$ and self-adjoint on $D(-\Delta)$.  We refer the reader to \cite{RS,Taylor} for more theory of this operator. The nonlinear equation \eqref{equ1.1} and many variations aspects have been studied extensively in the literature. In particular, the existence of a unique strong global-in-time solution to \eqref{equ1.1} with Hartree nonlinearity $f(|u|^2)=|x|^{-1}\ast|u|^2$ goes back to \cite{CG}.
When $K\leq 0$, the solution $u(t)$ to \eqref{equ1.1} with the Hartree nonlinearity is studied in \cite{DF, HO} in which they proved the global existence and a decay rate for the solution; however, they need
the initial data in a weighted-$L^2$ space. When $K>0$, Lenzmann and Lewin \cite{LeLe} proved a time average estimate holds for every $R>0$
such that
\begin{equation}
\limsup_{T\to\infty}\frac1{T}\int_0^T\int_{|x|\leq R}|u(t,x)|^2 dx dt\leq 4K
\end{equation}
and
\begin{equation}
\limsup_{T\to\infty}\frac1{T}\int_0^T\int_{|x|\leq R}|\nabla u(t,x)|^2 dx dt\leq K^3
\end{equation}
which is related to the RAGE theorem (see Reed-Simon\cite{RS}).

In this paper, we will study the Cauchy problem for the nonlinear Schr\"odinger equation \eqref{equ1.1} with initial data in energy space $H^1(\R^3)$.
The Cauchy problem, including the global existence and scattering theory, for the nonlinear Schr\"odinger equation without potential, i.e. $K=0$, has been intensively studied in \cite{Cav,GV79}.
Due to the perturbation of the long range potential, many basic tools which were used to study the nonlinear Schr\"odinger equation are different even fails.
We only have
a local-in-time Strichartz estimate and global-in-time Strichartz estimate fails when $K>0$,. We therefore show the solution of \eqref{equ1.1} is global existence but does not scatter.
Fortunately, in the case $K<0$, Mizutani \cite{Miz} recently obtained the global-in-time Strichartz estimate  by employing several techniques from scattering theory such as the long time parametrix
construction of Isozaki-Kitada type \cite{IK}, propagation estimates and local decay estimates. In this repulsive case, we will establish an interaction Morawetz estimate for
the  defocusing case, which provides us
a decay of the solution $u$ to \eqref{equ1.1}. Combining this with  the global-in-time Strichartz estimate \cite{Miz}, we therefore obtain the scattering theory in the repulsive and defocusing cases. It is worth  mentioning that in the proof of scattering theory,  we also need a chain rule which is
established by proving the equivalence of the Sobolev norm
from the heat kernel estimate, as we did in \cite{KMVZZ-Sobolev,ZZ}. Even though we obtain some results for
this Cauchy problem,  the whole picture of the nonlinear Schr\"odinger equation with the Coulomb potential is far to be completed, for example,
the scattering theory in the energy-critical cases.

Equation \eqref{equ1.1} admits a number of symmetries in  $H^1(\R^3)$, explicitly:

$\bullet$ {\bf Phase invariance:} if $u(t,x)$ solves \eqref{equ1.1}, then so does $e^{i\gamma}u(t,x),~\gamma\in\R;$

$\bullet$ {\bf Time translation invariance:} if $u(t,x)$ solves \eqref{equ1.1}, then so does $u(t+t_0,x+x_0),~(t_0,x_0)\in\R\times\R^3$.

From the Ehrenfest law or direct computation, these symmetries induce invariances in the energy space, namely:
mass
\begin{equation}\label{equ:mass}
M(u)=\int_{\R^3} |u(t,x)|^2\;dx=M(u_0)
\end{equation}
and energy
\begin{equation}\label{equ:energy}
E(u)=\int_{\R^3}\Big(\frac12|\nabla u|^2-\frac{K}2\frac{|u|^2}{|x|}+\frac{\lambda}{p+1}|u|^{p+1}\Big)\;dx.
\end{equation}
Comparing with the classical Schr\"odinger equation (i.e. \eqref{equ1.1} with $K=0$), equation \eqref{equ1.1} is not {\bf space translation invariance}, which induces that  the momentum
$$P(u):={\rm Im}\int_{\R^3}\bar{u}\nabla u\;dx$$
is not conserved.  Removing the potential term $\frac{K}{|x|}u$, one recovers the classical
nonlinear Schr\"odinger equation:
\begin{align} \label{equ:nls}
\begin{cases}    (i\partial_t+\Delta)u= \lambda |u|^{p-1}u,\quad
(t,x)\in\R\times\R^3,
\\
u(0,x)=u_0(x)\in H^1(\R^3),
\end{cases}
\end{align}
which is scaling invariant. That is,
the class of solutions to \eqref{equ:nls} is left invariant by the
scaling
\begin{equation}\label{scale}
u(t,x)\mapsto \mu^{\frac2{p-1}}u(\mu^2t, \mu
x),\quad\mu>0.
\end{equation}
Moreover, one can also check that the only homogeneous $L_x^2$-based
Sobolev space that is left invariant under \eqref{scale} is
$\dot{H}_x^{s_c}(\R^3)$ with $s_c:=\tfrac{3}2-\tfrac2{p-1}$. When $s_c<1$, the problem is called energy-subcritical problem.
The problem is known as energy-critical problem when $s_c=1$. There are a number of work to study the problems, we refer the reader to \cite{Bo99a,Cav,CKSTT07,GV85,RV,Visan2007} for defocusing case in the energy-subcritical and
energy-critical cases; to \cite{Dodson4,DM1,DM2,DHR,HR,KM,KV20101} for the focusing case.   It is known that the defocusing case is different from the focusing one due to the opposite sign between the kinetic energy and potential energy.

In this paper, we mainly consider the influence of the long range potential $K|x|^{-1}$ on the existence theory and scattering theory for nonlinear Schr\"odinger equation.
We will find  some  influences, e.g. global existence, are same as the result of  \eqref{equ:nls}; but, in particular $K>0$, some results are quite different. For example,  the solution is global existence no matter what sign of $K$, but
it scatters when $K<0$ but does not scatter  when $K>0$ even in the defocusing case.

As mentioned above the focusing case is different from the defocusing case. In the focusing case $(\lambda=-1)$, we will also use the energy without potential
$$E_0(u):=\frac12\int_{\R^3}|\nabla u(t,x)|^2\;dx-\frac1{p+1}\int_{\R^3}|u(t,x)|^{p+1}\;dx,$$
to give the threshold for global/blowup dichotomy.
As the same argument as in \cite{KMVZ,LMM} considering NLS with an inverse square potential, in the case $K<0,$  we will consider  the initial data below the threshold of the ground state $Q$ to the classical elliptic equation
\begin{equation}\label{equ:ground}
  -\Delta Q+Q=Q^p,\quad 1<p<5
\end{equation}
due to the sharp constant in the Gagliardo-Nirenberg inequality
\begin{equation}\label{equ:gni}
  \|f\|_{L^{p+1}}^{p+1}\leq C_K\|f\|_{L^2}^\frac{5-p}{2}\|\sqrt{\lk}f\|_{L^2}^\frac{3(p-1)}{2}= C_K\|f\|_{L^2}^\frac{5-p}{2}\Big(\|f\|_{\dot{H}^1}^2-K\int\tfrac{|f|^2}{|x|}\;dx\Big)^\frac{3(p-1)}{4}.
\end{equation}
Let $C_0$ be the sharp constant of the classical Gagliardo-Nirenberg inequality
\begin{equation}\label{equ:clagn}
  \|f\|_{L^{p+1}}^{p+1}\leq C_0\|f\|_{L^2}^\frac{5-p}{2}\|f\|_{\dot{H}^1}^\frac{3(p-1)}{2}.
\end{equation}
Then,
we claim that $C_K=C_0$, it is well-known that equality in \eqref{equ:clagn} with $K=0$ is attained by $Q$, but we will see that equality in \eqref{equ:gni} with $K<0$ is never attained. Indeed, by the sharp Gagliardo-Nirenberg inequality for  \eqref{equ:clagn}, we find
$$\lim_{n\to\infty}\frac{ \|Q\|_{L^{p+1}}^{p+1}}{\|Q\|_{L^2}^\frac{5-p}{2}\Big(\|Q\|_{\dot{H}^1}^2-K\int\tfrac{|Q|^2}{|x-n|}\;dx\Big)^\frac{3(p-1)}{4}}=\frac{ \|Q\|_{L^{p+1}}^{p+1}}{\|Q\|_{L^2}^\frac{5-p}{2}\|Q\|_{\dot{H}^1}^\frac{3(p-1)}{2}}=C_0.$$
Thus, $C_0\leq C_K.$ However, for any $f\in H^1\setminus\{0\}$ and $K<0$,
 the standard Gagliardo-Nirenberg inequality implies
$$\|f\|_{L^{p+1}}^{p+1}\leq C_0\|f\|_{L^2}^\frac{5-p}{2}\|f\|_{\dot{H}^1}^\frac{3(p-1)}{2}< C_0\|f\|_{L^2}^\frac{5-p}{2}\|\sqrt{\lk}f\|_{L^2}^\frac{3(p-1)}{2}.$$
Thus $C_K = C_0$, and the last estimate also shows that
equality is never attained.

In the energy-critical case ($s_c=1$), we consider
 the ground state $W$ to be the elliptic equation
$$-\Delta W=W^5$$
due to the sharp constant in Sobolev embedding.
We refer to \cite{BL,GNN,Kw}  about the existence and uniqueness of the ground state.\vspace{0.2cm}

Now, we state our main results. First, we consider the global well-posedness theory for the problem \eqref{equ1.1} under some restrictions.
In the energy-subcritical case (i.e $p-1<4$), the global well-posedness will follow from local well-posedness theory and uniform kinetic energy
control
\begin{equation}\label{equ:uniforkinetic1231}
  \sup_{t\in I}\|u(t)\|_{\dot{H}^1(\R^3)}\leq C(E(u_0),M(u_0)),
\end{equation}
And the local well-posedness will be proved by the standard fixed point argument combining with Strichartz estimate on Lorentz space.

In the energy-critical case ($p-1=4$),  we will show the global well-posedness
by controlling global kinetic energy \eqref{equ:uniforkinetic1231} and  proving ``good local
well-posedness". More precisely, using perturbation argument as in Zhang \cite{Zhang} and global well-posedness for equation \eqref{equ:nls} under some restrictions, we will show that there exists a
small constant $T=T(\|u_{0}\|_{H^{1}_{x}})$ such that \eqref{equ1.1}
is well-posed on $[0,T]$, which is so-called ``good local
well-posed". On the other hand, since the equation in \eqref{equ1.1}
is time translation invariant, this ``good local well-posed"
combining with the global kinetic energy control \eqref{equ:uniforkinetic1231} gives immediately
the global well-posedness. We remark that this argument also works for the energy-subcritical case.

\begin{theorem}[Global well-posedness]\label{thm:global} Let $K\in\R$ and $u_0\in H^1(\R^3)$. Suppose that
$0<p-1\leq4$ in the defocusing case $\lambda=1$. While for the focusing case $\lambda=-1$, we assume that $0<p-1<\frac43$ $($mass-subcritical$)$ or
\begin{itemize}
\item  If $p-1=\frac43$ $($mass-critical$)$ , assume $M(u_0)<M(Q)$.

\item If $\frac43<p-1<4$ and $K<0$, assume \footnote{For $K<0$ and $\lambda=-1$, we remark that under the assumption $M(u_0)^{1-s_c}E(u_0)^{s_c}<M(Q)^{1-s_c}E_0(Q)^{s_c}$, the condition
$\|u_0\|_{L^2}^{1-s_c}\|u_0\|_{\dot{H}^1}^{s_c}<\|Q\|_{L^2}^{1-s_c}\|Q\|_{\dot{H}^1}^{s_c}$
is equivalent to
$\|u_0\|_{L^2}^{1-s_c}\Big(\|u_0\|_{\dot{H}^1}^2-K\big\||x|^{-\frac12}u_0\big\|_{L^2}^2\Big)^{\frac{s_c}2}<\|Q\|_{L^2}^{1-s_c}\|Q\|_{\dot{H}^1}^{s_c}.$
See Remark \ref{rem:tde}.}
 \begin{equation}\label{equ:mecond}
   M(u_0)^{1-s_c}E(u_0)^{s_c}<M(Q)^{1-s_c}E_0(Q)^{s_c},\; \|u_0\|_{L^2}^{1-s_c}\|u_0\|_{\dot{H}^1}^{s_c}<\|Q\|_{L^2}^{1-s_c}\|Q\|_{\dot{H}^1}^{s_c}.
 \end{equation}

\item  If $p-1=4$ $($energy-critical$)$   and $K<0$, assume that $u_0$ is radial\footnote{Here the restriction $K<0$ induces us to utilize the result of Kenig-Merle \cite{KM} in which one needs a radial initial data. } and
 \begin{equation}\label{equ:energythre}
   E(u_0)<E_0(W),\; \|u_0\|_{\dot{H}^1}<\|W\|_{\dot{H}^1}.
 \end{equation}
\end{itemize}
Then, there exists a unique global solution $u(t,x)$ to \eqref{equ1.1} such that
\begin{equation}\label{equ:ugolbaunif}
  \|u\|_{L_t^q(I,H^{1,r})}\leq C(\|u_0\|_{H^1},|I|),
\end{equation}
for any $I\subset\R$ compact and $(q, r)\in\Lambda_0$ admissible defined below.
\end{theorem}

\begin{remark}The global existence is almost completed in the defocusing case regardless of whether in the repulsive or attractive case.
The focusing case is more complicated and the following blow up result below is a supplement of this global existence.
\end{remark}
Next, for the global solution $u$ to equation \eqref{equ1.1}, we  want to study the long-time behavior of the solution, such as scattering theory. We say that a global solution $u$ to \eqref{equ1.1}  \emph{scatters},
if there exist $u_\pm\in H_x^1(\R^3)$ such that
\[
\lim_{t\to\pm\infty} \| u(t) - e^{-it\lk}u_{\pm} \|_{H_x^1(\R^3)} =
0.
\]

From the argument as in the proof of well-posedness theory, we know that one can regard the long-range potential term $\frac{K}{|x|}u$ as the nonlinear perturbation term(it looks like the cubic nonlinear term $|u|^2u$ from scaling analysis). However, by Reed-Simon\cite{RS}, we know that the limits
$$s-\lim_{t\to\pm\infty}e^{it\lk}e^{it\Delta}\quad \text{in}\quad L^2(\R^3)$$
do not exist. Therefore,  we can not regard  the potential term $\frac{K}{|x|}u$ as the nonlinear perturbation in the scattering theory. We refer the reader to several different constructions of wave operators in the long-range case, such as momentum approach\cite{Hor}, Isozaki-Kitada method\cite{IK} and position approach \cite{DG,Ya}.

On the other hand, the standard arguments show that the scattering is equivalent to the global Strichartz-norm boundedness ($\|u(t)\|_{L_t^q(\R;L_x^r(\R^3))}<+\infty)$ provided that we have the global-in time Strichartz estimate. However, in the attractive case, i.e. $K>0$, the global-in-time Strichartz estimate does not hold, see Subsection \ref{sub:stricharz} below. Thus,
we don't know whether the solution $u$ to \eqref{equ1.1} with $K>0$ scatters or not even for the small initial data. While for the  repulsive case, i.e $K<0$, the global-in-time Strichartz estimates were recently established by  Mizutani \cite{Miz}. Then, combining with  Sobolev norm equivalence \eqref{equ:sobequi123} below, one can easily obtain the scattering result for the small initial data. For the general initial data, we will get  the scattering result in the defocusing energy-subcritical case ($\lambda=1,~p<5$) by establishing the interaction Morawetz estimate, which gives a  global Strichartz-norm boundedness.

In the case $K>0$, we know from \cite[Lemma 6]{BJ} that there is a positive solution $f(x)\in H^2$ of the elliptic equation
\begin{equation}\label{equ:by}
  -\Delta f-\frac{K}{|x|}f+f+f^p=0.
\end{equation}
This implies that there is a soliton $u(t,x):=e^{it}f(x)$ solves \eqref{equ1.1} with $\lambda=1$.
We remark that such soliton is global but not scatters.
Equation \eqref{equ:by} arises in the Thomas-Fermi-von Weizsacker (TFW) theory of
atoms and molecules \cite{BBL,Lieb} without electronic repulsion. There, $K|x|^{-1}$ is the
electric potential due to a fixed nucleus of atomic number $K$ located at the origin,
$f(x)^2$ stands for the electronic density and $\int f(x)^2dx$ is the total number of
electrons.

While for the case $K\leq0$, we will derive the quadratic Morawetz indentity for \eqref{equ1.1} and then
establish the following interaction Morawetz estimate for $\lambda=1$
\begin{equation}\label{equ:intmorest}
  \int_{\R}\int_{\R^3}|u(t,x)|^4\;dx\;dt\leq CM(u_0)\sup_{t\in\R}\|u(t,\cdot)\|_{\dot{H}^\frac12}^2,
\end{equation}
which provides us
a decay of the solution $u$ to \eqref{equ1.1}.
Combining this with Strichartz estimate and
Leibniz rule obtained by the following Sobolev norm equivalence
\begin{equation}\label{equ:sobequi123}
  \big\|\sqrt{1+\lk}f\big\|_{L^p(\R^3)}\simeq \big\| \sqrt{1-\Delta}f\big\|_{L^p(\R^3)},\quad 1<p<3,
\end{equation}
we establish the scattering theory as follows.
\begin{theorem}[Scattering theory]\label{thm:scattering}
Let $K\leq0,\; \frac43<p-1<4,\;\lambda=1$ and $u_0\in H^1(\R^3)$. Then, there exists a global solution $u$ to \eqref{equ1.1}, and the solution
$u$ scatters in the sense that there exists $u_{\pm}\in H^1(\R^3)$ such that
\begin{equation}\label{equ:uscat}
  \lim_{t\to\pm\infty}\big\|u(t,\cdot)-e^{-it\lk}u_{\pm}\big\|_{H^1(\R^3)}=0.
\end{equation}
\end{theorem}

%
%
%

In the focusing case, i.e $\lambda=-1$, by the classical Virial argument, one can obtain the blow-up result for the negative energy.
\begin{theorem}[Blow-up result]\label{thm:blowup}
Let $K\in\R,\; \frac43< p-1\leq4,\;\lambda=-1$.

$(i)$
Let $u_0\in \Sigma:=\{u_0\in H^1,\; xu_0\in L^2\}$.   Then, the solution $u$ to \eqref{equ1.1} blows up in both time direction, in one of the three cases:
\begin{enumerate}
  \item $C(E(u_0),M(u_0))<0$;

\item $C(E(u_0),M(u_0))=0,\; y'(0)<0$;

\item $C(E(u_0),M(u_0))>0,\; y'(0)^2\geq 24(p-1) C(E(u_0),M(u_0))\big\||x|u_0\big\|_{L^2(\R^3)}^2$;
\end{enumerate}
where
$$y'(0)=4{\rm Im}\int_{\R^3}x\cdot\nabla u_0\bar{u}_0\;dx,$$
and
\begin{equation}\label{equ:cem}
  C(E(u_0),M(u_0)):= \begin{cases}
  E(u_0) \quad\text{if}\quad K\leq0\\
  E(u_0)+\frac{3K^2}{2(3p-7)(p-1)}M(u_0)\quad\text{if}\quad K>0.
  \end{cases}
\end{equation}

$(ii)$ Let $u_0\in H^1(\R^3)$ be radial, and
assume that $C(E(u_0),M(u_0))<0.$ Then, the solution $u$ to \eqref{equ1.1} blows up in both time direction.

\end{theorem}

The paper is organized as follows. In Section $2$,  as a
preliminaries, we give some notation, recall the Strichartz
estimate and prove the Sobolev space equivalence. Section $3$ is
devoted to proving global well-posedness, i.e Theorem \ref{thm:global}.  We show the interaction Morawetz-type estimates in Section $4$, and we utilize such Morawetz-type estimates and the equivalence of
Sobolev norm to prove Theorem \ref{thm:scattering}. Finally, we use the Virial argument to obtain the blow-up result (Theorem \ref{thm:blowup}) in Section 5.

\subsection*{Acknowledgements} The authors were supported by NSFC Grants 11771041, 11831004.    We are grateful to R. Killip, J. Murphy and M. Visan for useful discussions.




\section{Preliminaries}

In this section, we first introduce some notation, and then recall the Strichartz estimates.
 We conclude this section by showing the Sobolev space equivalence between the operator $\lk$ and
Laplacian operator $-\Delta$.

\subsection{Notations}
First, we give some notations which will be used throughout this
paper. To simplify the expression of our inequalities, we introduce
some symbols $\lesssim, \thicksim, \ll$. If $X, Y$ are nonnegative
quantities, we use $X\lesssim Y $ or $X=O(Y)$ to denote the estimate
$X\leq CY$ for some $C$, and $X \thicksim Y$ to denote the estimate
$X\lesssim Y\lesssim X$.   We
denote $a_{\pm}$ to be any quantity of the form $a\pm\epsilon$ for
any $\epsilon>0$.

For a spacetime slab $I\times\R^3$, we write $L_t^q L_x^r(I\times\R^3)$ for the Banach space of functions $u:I\times\R^3\to\C$ equipped with the norm
    $$\|u\|_{L_t^q(I;L_x^r(\R^3))}:=\bigg(\int_I \|u(t,\cdot)\|_{L_x^r(\R^3)}\bigg)^{1/q},$$
with the usual adjustments when $q$ or $r$ is infinity. When $q=r$, we abbreviate $L_t^qL_x^q=L_{t,x}^q$.
We will also often abbreviate $\|f\|_{L_x^r(\R^3)}$ to $\|f\|_{L_x^r}.$ For $1\leq r\leq\infty$,
we use $r'$ to denote the dual exponent to $r$, i.e. the solution to $\tfrac{1}{r}+\tfrac{1}{r'}=1.$

The Fourier transform on $\mathbb{R}^3$ is defined by
\begin{equation*}
\aligned \widehat{f}(\xi):= \big( 2\pi
\big)^{-\frac{3}{2}}\int_{\mathbb{R}^3}e^{- ix\cdot \xi}f(x)dx ,
\endaligned
\end{equation*}
giving rise to the fractional differentiation operators
$|\nabla|^{s}$ and $\langle\nabla\rangle^s$,  defined by
\begin{equation*}
\aligned
\widehat{|\nabla|^sf}(\xi):=|\xi|^s\hat{f}(\xi),~~\widehat{\langle\nabla\rangle^sf}(\xi):=\langle\xi\rangle^s\hat{f}(\xi),
\endaligned
\end{equation*} where $\langle\xi\rangle:=1+|\xi|$.
This helps us to define the homogeneous and inhomogeneous Sobolev
norms
$$\|u\|_{\dot{W}^{s,p}(\R^3)}=\big\||\nabla|^su\big\|_{L^p},\; \|u\|_{{W}^{s,p}(\R^3)}=\big\|\langle\nabla\rangle^su\big\|_{L^p}.$$
Especially, for $p=2$, we denote $\dot{W}^{s,p}(\R^3)=\dot{H}^s(\R^3)$ and ${W}^{s,p}(\R^3)=H^s(\R^3).$

Next, we recall the well-known Lorentz space and some properties of this space for our purpose.
Given a measurable function $f: \R^3\to \C$, define the distribution function of $f$ as
$$f_\ast(t)=\mu(\{x\in \R^3: |f(x)|> t\}),\quad t>0$$ and its rearrangement function as $$f^*(s)=\inf\{t: f_\ast(t)\leq s\}.$$ For $1\leq p<\infty$ and $1\leq r\leq \infty$, define the Lorentz quasi-norm
\begin{equation*}
\|f\|_{L^{p,r}(\R^3)}=\begin{cases}\Big(\int_0^\infty(s^{\frac1p}f^*(s))^r\frac{ds}{s}\Big)^{1/r}, &\quad 1\leq r<\infty;\\
\sup\limits_{s>0} s^{\frac1p}f^*(s),&\qquad r=\infty.
\end{cases}
\end{equation*}
The Lorentz space  $L^{p,r}(\R^3)$ denotes the space of complex-valued measurable functions $f$ on $\R^3$ such that its quasi-norm $\|f\|_{L^{p,r}(\R^3)}$ is finite.
From this characterization, $L^{p,\infty}(\R^3)$ is the usual weak $L^p$ space, $L^{p,p}(\R^3)=L^p(\R^3)$ and $L^{p,r}(\R^3)\subset L^{p,\tilde{r}}(\R^3)$ with $r<\tilde{r}$.

We refer to O'Neil \cite{Neil} for the following H\"older inequality in Lorentz space.
\begin{proposition}[H\"older's inequality in Lorentz space]\label{Lorentz} Let $1\leq p, p_0, p_1<\infty$ and $1\leq r, r_0, r_1\leq \infty$, then
\begin{equation}
\|fg\|_{L^{p,r}}\leq C \|f\|_{L^{p_0,r_0}} \|g\|_{L^{p_1,r_1}}, \quad \frac1p=\frac{1}{p_0}+\frac{1}{p_1}, ~~\frac1r=\frac{1}{r_0}+\frac{1}{r_1}.
\end{equation}
\end{proposition}

\subsection{Strichartz estimate}\label{sub:stricharz} It is well known that the Strichartz estimate is very useful in the study of the
nonlinear dispersive equations. To state the result, we  define
\begin{equation}
\Lambda_0=\big\{(q,r):\;\tfrac2q=3\big(\tfrac12-\tfrac1r\big), q,r\geq2\big\}.
\end{equation}

\begin{theorem}[Local-in-time Strichartz estimate]\label{thm:locastri} Let $K\in\R$ and $\lk$ be as above.
For $(q,r)\in\Lambda_0$, there holds
\begin{equation}\label{equ:locainstr}
\|e^{it\mathcal{L}_K}f\|_{L_t^q(I,L_x^r)}\leq C(|I|)\|f\|_{L_x^2}.
\end{equation}

\end{theorem}

%
%
%

\begin{proof} The proof is based on a perturbation argument.  Let $u(t,x)=e^{it\mathcal{L}_K}f$, then $u$ satisfies that
$$i\pa_tu+\Delta u=-\frac{K}{|x|}u,\;u(0,x)=f(x)$$
We regard the Coulomb potential as an inhomogeneous term, hence we have by Duhamel's formula
$$e^{it\mathcal{L}_K}f=u(t)=e^{it\Delta}f+iK\int_0^te^{i(t-s)\Delta}\frac{u}{|x|}\;dx.$$
For our purpose, we recall the inhomogeneous Strichartz estimate without potential on Lorentz space.
\begin{lemma}[Strichartz estimate for $e^{it\Delta}$, \cite{KT,Plan}]\label{lem:locinstr}
For $(q,r),(q_1,r_1)\in \Lambda_0$, we have
\begin{equation}\label{equ:edstrlor}
\begin{split}
\|e^{it\Delta}f\|_{L_t^q(I,L_x^r)}&\leq C\|f\|_{L_x^2};\\
\Big\|\int_0^t e^{i(t-s)\Delta}F(s)\;ds\Big\|_{L_t^q(I,L_x^{r,2})}&\leq C\|F(t,x)\|_{L_t^{q_1'}(I,L_x^{r_1',2})},
\end{split}
\end{equation}
where $\frac1q+\frac{1}{q'}=1$.

\end{lemma}
Using the above lemma, we show that
\begin{align*}
\|e^{it\lk}f\|_{L_t^q(I,L_x^{r})}\leq C\|f\|_{L^2(\R^3)}+|K|\Big\|\int_0^te^{i(t-s)\Delta}\frac{u}{|x|}\;dx\Big\|_{L_t^q(I,L_x^{r})}.
\end{align*}
We use the above inhomogeneous Strichartz estimate to obtain
\begin{align*}
&\Big\|\int_0^te^{i(t-s)\Delta}\frac{u}{|x|}\;dx\Big\|_{L_t^q(I,L_x^{r})}\leq
\Big\|\int_0^te^{i(t-s)\Delta}\frac{u}{|x|}\;dx\Big\|_{L_t^q(I,L_x^{r,2})}\\\leq &
C\Big\|\frac{u}{|x|}\Big\|_{L_t^2(I,L_x^{\frac{6}{5},2})}
\leq C|I|^\frac12\||x|^{-1}\|_{L_x^{3,\infty}}\|u\|_{L_t^\infty L_x^2}\\
\leq&C(|I|)\|f\|_{L_x^2}
\end{align*}
where we use the mass conservation in the last inequality. Therefore we prove \eqref{equ:locainstr}.
\end{proof}

It is nature to ask whether the global-in-time Strichartz estimate holds or not. The answer is that
the global-in-time Strichartz estimate does not hold in the attractive case $K>0$ but holds
in the repulsive case $K\leq 0$.\vspace{0.2cm}

To see the attractive case, a simple computation shows
$$\Delta(e^{-c|x|})=c^2e^{-c|x|}-\frac{2}{|x|}ce^{-c|x|}.$$
Let $c_K=K/2$, this implies
$$\lk(e^{-c_K|x|})=\Big(-\Delta-\frac{K}{|x|}\Big)(e^{-c_K|x|})=-\Big(\frac{K}{2}\Big)^2e^{-c_K|x|}.$$
Then, the function $u(t,x)=e^{ic_K^2t}(e^{-c_K|x|})$ with $c_K=\frac{K}{2}$ solves the linear equation $i\pa_tu-\mathcal{L}_Ku=0$ and $u_0(x)=e^{-c_K|x|}\in L^2(\R^3)$ when $K>0$. However,
\begin{equation}
\|u(t,x)\|_{L_t^q(\R,L_x^r(\R^3))}=+\infty.
\end{equation}

In the repulsive Coulomb potential case, Mizutani \cite{Miz} recently proved the global-in-time Strichartz estimate, where the proof employs several techniques from linear scattering theory such as the long time parametrix
construction of Isozaki-Kitada type \cite{IK}, propagation estimates and local decay estimates.
\begin{theorem}[Global-in-time Strichartz estimate,\cite{Miz}]\label{thm:globalstri}
For $(q,r),(q_1,r_1)\in\Lambda_0$ and $K<0$, there holds
\begin{equation}\label{equ:globainstr}
\|e^{it\mathcal{L}_K}f\|_{L_t^q(\R,L_x^r)}\leq C\|f\|_{L_x^2},
\end{equation}
and
\begin{equation}\label{equ:globainstrinh}
\Big\|\int_0^t e^{i(t-s)\mathcal{L}_K}F(s)\;ds\Big\|_{L_t^q(\R,L_x^r)}\leq C\|F\|_{L_t^{q_1'}(\R,L_x^{r_1'})}.
\end{equation}
\end{theorem}

\subsection{Fractional product rule} As mentioned in the introduction, we need the following fractional chain rule in the proof of scattering theory when $K<0$.
The $L^p$-product rule for fractional derivatives in Euclidean spaces
\begin{align*}
\| (-\Delta)^{\frac s2}(fg)\|_{L^p(\R^3)} \lesssim &\| (-\Delta)^{\frac s2} f\|_{L^{p_1}(\R^3)}\|g\|_{L^{p_2}(\R^3)}\\
&+\|f\|_{L^{q_1}(\R^3)}\| (-\Delta)^{\frac s2} g\|_{L^{q_2}(\R^3)},
\end{align*}
was first proved by Christ and Weinstein \cite{ChW}. Here $1<p, p_1, p_2, q_1, q_2< \infty$, $s\geq0$ and $\frac1p=\frac1{p_1}+\frac1{p_2}=\frac1{q_1}+\frac1{q_2}$.
Similarly, we have
the following for the operator $\lk$ with $K<0$.

\begin{lemma}[Fractional product rule]\label{L:Leibnitz}
Fix $K<0$ and let $\lk$ be as above.  Then for all $f, g\in C_c^{\infty}(\R^3\setminus\{0\})$ we have
\begin{align*}
\| \sqrt{1+\lk}(fg)\|_{L^p(\R^3)} \lesssim \| \sqrt{1+\lk} f\|_{L^{p_1}(\R^3)}\|g\|_{L^{p_2}(\R^3)}+\|f\|_{L^{q_1}(\R^3)}\| \sqrt{1+\lk} g\|_{L^{q_2}(\R^3)},
\end{align*}
for any exponents satisfying $1< p, p_1, q_2< 3$, $1<p_2, q_1 <\infty$ and $\frac1p=\frac1{p_1}+\frac1{p_2}=\frac1{q_1}+\frac1{q_2}$.
\end{lemma}

This is a consequence of the equivalence of Sobolev norm
$$
\big\| f\big\|_{L^p(\R^3)}+\big\|\nabla f\big\|_{L^p(\R^3)} \sim \big\| \sqrt{1+\lk} f\big\|_{L^p(\R^3)},\quad   1< p < 3.
$$
which will be proved in the next subsection.

\subsection{Sobolev space equivalence}

In this subsection, we study the relationship between  Sobolev space adapted with Laplacian operator perturbed by Coulomb potential and classical Laplacian operator, that is,
for suitable $s$ and $p$ such that

\begin{equation}\label{equ:soalarge}
  \big\|\langle \lk\rangle^\frac{s}{2} f\big\|_{L^p(\R^3)}\simeq \big\|\langle\nabla\rangle^\frac{s}{2} f\big\|_{L^p(\R^3)}
\end{equation}
where $\langle a\rangle=(1+|a|^2)^{1/2}$. To this end, we recall the heat kernel estimate

\begin{lemma}[Heat kernel]\label{lem:heat} Let $K<0$ and let $\lk$ be as above. Then there exist constants $C,c>0$ such that
\begin{equation}\label{equ:heatk}
0\leq  e^{-t\mathcal{L}_K}(x,y)\leq Ct^{-3/2}e^{-\frac{|x-y|^2}{ct}}.
\end{equation}

\end{lemma}

\begin{proof} Since $K<0$, then $\lk=-\Delta+V(x)$ with a positive positive $V=-K|x|^{-1}$. It is easy to verify that $V\in L^2_{\text{loc}}(\R^3)$.
It is well known that \eqref{equ:heatk}, e.g.  see \cite{Kurata}. Indeed, one can use the estimate of the fundamental solution of the elliptic operator $\lk+\lambda$ with non-negative parameter $\lambda$
in Shen \cite{Shen} to obtain the heat kernel estimate.

\end{proof}

\begin{lemma}[Sobolev norm equivalence]\label{lem:sobnorequ}
Let $K<0$, $1<p<3$ and $0\leq s\leq2.$ There holds
\begin{equation}\label{equ:sobequi}
  \big\|(1+\lk)^\frac{s}2f\big\|_{L^p(\R^3)}\simeq \big\| (1-\Delta)^\frac{s}2f\big\|_{L^p(\R^3)}.
\end{equation}
\end{lemma}

\begin{proof}
The proof is classical and follows from heat kernel estimate and Stein complex interpolation. We refer to Y. Hong\cite{Hong} or the authors \cite{ZZ}, but we give a complete
proof for convenience.

  First, we consider $s=2$. Using the Hardy inequality \cite[Lemma 2.6] {ZZ} with $p<3$, we obtain
  \begin{align*}
    \big\|(1+\lk)f\big\|_{L^p} \leq& \big\| (1-\Delta)f\big\|_{L^p}+|K|\big\|\tfrac{f}{|x|}\big\|_{L^p} \\
    \lesssim & \big\| (1-\Delta)f\big\|_{L^p}+\|\nabla f\|_{L^p}\\
    \lesssim& \big\| (1-\Delta)f\big\|_{L^p}.
  \end{align*}
 By Lemma \ref{lem:heat}, we see the heat kernel operator $e^{-t(1+\lk)}$ obeys the Gaussian heat kernel estimate. Hence we easily get the Hardy's inequality for $p<3$
  $$\big\|\tfrac{f}{|x|}\big\|_{L^p}\lesssim\big\|\sqrt{1+\lk}f\big\|_{L^p}.$$
  Hence,
  \begin{align*}
    \big\|(1-\Delta)f\big\|_{L^p} \leq& \big\| (1+\lk)f\big\|_{L^p}+|K|\big\|\tfrac{f}{|x|}\big\|_{L^p}\\
    \lesssim&\big\| (1+\lk)f\big\|_{L^p}+\big\| \sqrt{1+\lk}f\big\|_{L^p}\\
    \lesssim&\big\| (1+\lk)f\big\|_{L^p}.
  \end{align*}
  This implies \eqref{equ:sobequi} with $s=2$.

Next, since the heat kernel operator $e^{-t(1+\lk)}$ obeys the Gaussian heat kernel estimate, we have by Sikora-Wright \cite{SW}
$$\big\|(1-\Delta)^{ib}f\big\|_{L^p}+\big\|(1+\lk)^{ib}f\big\|_{L^p}\lesssim\langle b\rangle^\frac32,\quad \forall~b\in\R,\;\forall~1<p<+\infty.$$
Let $z=a+ib$, define
$$T_z=(1+\lk)^z(1-\Delta)^{-z}, G_z=(1-\Delta)^z(1+\lk)^{-z}.$$
Then we have that for $1<p<3$
$$\|T_{1+ib}\|_{L^p\to L^p}\leq \langle b\rangle^3\|(1+\lk)(1-\Delta)^{-1}\|_{L^p\to L^p}\leq C\langle b\rangle^3.$$
This shows that
\begin{align*}
  \big\|(1-\Delta)^zf\big\|_{L^p}\lesssim &\langle {\rm Im } z\rangle^\frac32\big\|(1+\lk)^zf\big\|_{L^p} \\
 \big\|(1+\lk)^zf\big\|_{L^p}\lesssim &\langle {\rm Im } z\rangle^\frac32\big\|(1-\Delta)^zf\big\|_{L^p}
\end{align*}
holds for $1<p<+\infty$ when ${\rm Re}z=0$ and for $1<p<3$ when ${\rm Re z}=1$.  Therefore, \eqref{equ:sobequi} follows by the Stein complex interpolation.

\end{proof}

\section{Global well-posedness}
In this section, we prove the well-posedness for equation \eqref{equ1.1} including local and global well-posedness.
In this part, we only use the classical Strichartz estimate for the Schr\"odinger equation without potential $i\pa_tu-\Delta u=0$ on Lorentz space.

In the energy-subcritical case (i.e $p-1<4$), the global well-posedness will follow from local well-posedness theory and uniform kinetic energy
control
\begin{equation}\label{equ:uniforkinetic123}
  \sup_{t\in I}\|u(t)\|_{\dot{H}^1(\R^3)}\leq C(E(u_0),M(u_0)).
\end{equation}

In the energy-critical case ($p-1=4$),  we prove the global well-posedness by using a perturbation argument and the
well-known scattering theory for Schr\"odinger without potential in \cite{CKSTT07, KM}.

\subsection{Local well-posedness for energy-subcritical: $s_c<1$}

\begin{theorem}[Local well-posedness, energy-subcritical]\label{thm:local} Let $K\in\R$, $0<p-1<4$ and $u_0\in H^1(\R^3)$.  Then
there exists $T=T(\|u_0\|_{H^1})>0$ such that the equation
\eqref{equ1.1} with initial data $u_0$ has a unique solution $u$ with
\begin{equation}\label{small}
u\in C(I; H^1(\R^3))\cap L_t^{q_0}(I,W^{1,r_0}(\R^3)),\quad I=[0,T],
\end{equation}
where $(q_0,r_0)=\big(\tfrac{4(p+1)}{3(p-1)},p+1\big)\in\Lambda_0.$
\end{theorem}

\begin{proof}

Define the map
\begin{equation}
\Phi(u(t)):=e^{it\Delta}u_0+i\int_0^t e^{i(t-s)\Delta}\Big(\frac{K}{|x|}u-\lambda|u|^{p-1}u\Big)(s)\;ds,
\end{equation}
with $I=[0,T]$
$$B(I)=\big\{u\in Y(I)=C(I,H^1(\R^3))\cap L_t^{q_0}(I,W^{1,r_0}),\; \|u\|_{Y(I)}\leq 2C\|u_0\|_{H^1}\big\},$$
and the metric
$d(u,v)=\|u-v\|_{L_t^{q_0}(I,L_x^{r_0})\cap L_t^\infty(I,L_x^2)}.$

For $u\in B(I)$, we have by Strichartz estimate \eqref{equ:edstrlor}
\begin{align*}
\big\|\Phi(u)\big\|_{Y(I)}\leq& C\|u_0\|_{H^1}+C\Big\|\langle \nabla\rangle \Big(\tfrac{K}{|x|}u\Big)\Big\|_{L_t^2 L_x^{\frac65,2}}+C\big\|\langle \nabla\rangle (|u|^{p-1}u)\big\|_{L_t^{q_0'} L_x^{r_0'}(I\times\R^3)}\\
\leq&C\|u_0\|_{H^1}+C_1T^\frac12\|u\|_{L_t^\infty H^1}+C_1T^{1-\frac{2}{q_0}}
  \|u\|_{L_t^\infty(I, L_x^{r_0})}^{p-1}\|u\|_{L_t^{q_0}(I,W^{1,r_0})}\\
  \leq&C\|u_0\|_{H^1}+C_1T^\frac12\|u\|_{L_t^\infty H^1}+C_1T^{1-\frac{2}{q_0}}\|u\|_{Y(I)}^p\\
\leq&C\|u_0\|_{H^1}+2CC_1T^\frac12\|u_0\|_{H^1}+2CC_1T^{\frac{5-p}{2(p+1)}}\|u_0\|_{H^1}(2C\|u_0\|_{H^1})^{p-1}\\
\leq&2C\|u_0\|_{H^1}
\end{align*}
by taking $T$ small such that
$$2C_1T^\frac12+
2C_1T^{\frac{5-p}{2(p+1)}}(2C\|u_0\|_{H^1})^{p-1}\leq1.$$

On the other hand, for $u,v\in B(I)$, we get by Strichartz estimate
\begin{align*}
d\big(\Phi(u),\Phi(v)\big)=&\Big\|\int_0^te^{i(t-s)\Delta}\big[\frac{K}{|x|}(u-v)-(|u|^{p-1}u-|v|^{p-1}v)\big](s)\;ds\Big\|_{L_t^{q_0}(I,L_x^{r_0})}\\
\leq&C\Big\|\frac{u-v}{|x|}\Big\|_{L_t^2 L_x^{\frac65,2}}+C\big\||u|^{p-1}u-|v|^{p-1}v\big\|_{L_t^{q_0'}(I,L_x^{r_0'})}\\
\leq& CT^\frac12\|u-v\|_{L_t^\infty L_x^2}+CT^{\frac{5-p}{2(p+1)}}\|u-v\|_{L_t^{q_0}(I,L_x^{r_0})}\big\|(u,v)\big\|_{L_t^\infty(I,H^1)}^{p-1}\\
\leq&\frac12d(u,v)
\end{align*}
by taking $T$ small such that
$$CT^\frac12+4CT^{\frac{5-p}{2(p+1)}}(2C\|u_0\|_{H^1})^{p-1}\leq\frac12.$$

A standard fixed point argument gives a unique local solution $u:[0,T]\times\R^3\to\C$ to \eqref{equ1.1}.

\end{proof}

\subsection{Global well-posedness for energy-subcritical: $s_c<1$}\label{sub:global} By the local well-posedness theory and mass conservation, the global well-posedness will follow from the uniform kinetic energy control
\begin{equation}\label{equ:uniforkinetic}
  \sup_{t\in I}\|u(t)\|_{\dot{H}^1(\R^3)}\leq C(E(u_0),M(u_0)).
\end{equation}
We argue the following several cases.\vspace{0.2cm}

{\bf Case 1: the defocusing case, i.e. $\lambda=1$.} In the defocusing case, we have the uniform bound
\begin{equation}\label{equ:unibode}
\|u(t,\cdot)\|_{H^1_x(\R^3)}\leq C(M(u_0),E(u_0)).
\end{equation}
In fact, we have by Hardy's inequality and Young's ineqaulity
\begin{equation}\label{equ:hardyi}
\int_{\R^3}\frac{|u|^2}{|x|}\;dx\leq C\|u\|_{\dot{H}^\frac12}^2\leq C\|u\|_{L^2_x}\|u\|_{\dot{H}^1}\leq \frac{1}{2|K|}\|u\|_{\dot{H}^1}^2+2C^2|K|\cdot\|u\|_{L^2_x}^2,
\end{equation}
which implies
\begin{align*}
E(u_0)=E(u)\geq&\frac14\int_{\R^3}|\nabla u(t)|^2\;dx-C^2|K| M(u_0)
\end{align*}
and hence
$$\|u(t)\|_{H^1}^2\leq C_1M(u_0)+4E(u_0).$$
Therefore we can extend the local existence to be a global one.\vspace{0.2cm}

{\bf  Case 2: $\lambda=-1, 0<p-1<\frac43$.}  In this case, we have by Gagliardo-Nirenberg inequality and Young's inequality
$$\|u\|_{L_x^{p+1}}^{p+1}\leq C\|u\|_{L_x^2}^\frac{5-p}{2}\|u\|_{\dot{H}^1}^\frac{3(p-1)}{2}\leq C_1M(u_0)^\frac{5-p}{7-3p}+\frac{p+1}8
\|u\|_{\dot{H}^1}^2.$$
This together with \eqref{equ:hardyi} implies
\begin{align*}
E(u_0)=E(u)\geq&\frac18\int_{\R^3}|\nabla u(t)|^2\;dx-C^2|K| M(u_0)-\frac{C_1}{p+1}M(u_0)^\frac{5-p}{7-3p},
\end{align*}
and so
$$\|u(t)\|_{H^1}^2\leq C_1M(u_0)+8E(u_0).$$
Thus we can obtain the global existence by extending the local solution.\vspace{0.2cm}

{\bf Case 3: $\lambda=-1,\;p=\frac73,\;M(u_0)<M(Q).$} For the mass-critical equation:
$$i\pa_tu+\Delta u+\frac{K}{|x|}u+|u|^\frac{4}{3}u=0.$$
From \eqref{equ:hardyi}, we obtain
$$\Big|\frac{K}2\int_{\R^3}\frac{|u|^2}{|x|}\;dx\Big|\leq \frac{\varepsilon}{2}\|\nabla u\|_{L^2}^2+\frac{C^2|K|}{\varepsilon}M(u_0).$$
One the other hand, we have by the sharp Gagliardo-Nirenberg inequality
$$\frac{3}{10}\int_{\R^3}|u|^\frac{10}{3}\;dx\leq\frac12\Big(\frac{\|u\|_{L^2}}{\|Q\|_{L^2}}\Big)^\frac{4}{3}\|\nabla u\|_{L^2}^2.$$
Hence,
\begin{align*}
E(u_0)\geq& \frac12\|\nabla u(t,\cdot)\|_{L^2}^2\Big(1-\varepsilon-\Big(\frac{\|u\|_{L^2}}{\|Q\|_{L^2}}\Big)^\frac{4}{3}\Big)-\frac{C^2|K|}{\varepsilon}M(u_0)\\
\geq&\frac14\|\nabla u(t,\cdot)\|_{L^2}^2\Big(1-\Big(\frac{\|u\|_{L^2}}{\|Q\|_{L^2}}\Big)^\frac{4}{3}\Big)-\frac{C^2|K|}{\varepsilon}M(u_0).
\end{align*}
This shows
$$\|u(t,\cdot)\|_{L_t^\infty H^1_x}\leq C(M(u_0),E(u_0)).$$\vspace{0.2cm}

{\bf Case 4: $\lambda=-1,\; K<0,\; \frac43<p-1<4.$} In this case, we assume that
$$  M(u_0)^{1-s_c}E(u_0)^{s_c}<M(Q)^{1-s_c}E_0(Q)^{s_c},\; \|u_0\|_{L^2}^{1-s_c}\|u_0\|_{\dot{H}^1}^{s_c}<\|Q\|_{L^2}^{1-s_c}\|Q\|_{\dot{H}^1}^{s_c}.$$
Then, there exists $\delta>0$ such that
$$M(u_0)^{1-s_c}E(u_0)^{s_c}\leq(1-\delta)M(Q)^{1-s_c}E_0(Q)^{s_c}.$$
By the sharp Gagliardo-Nirenberg inequality, we have
\begin{equation}\label{equ:gnineq}
\|f\|_{L_x^{p+1}}^{p+1}\leq C_0\|f\|_{L_x^2}^{\frac{5-p}2}
\|f\|_{\dot H^1}^{\frac{3(p-1)}2},
\end{equation}
with the sharp constant
\begin{equation}\label{equ:qah1}
C_0\|Q\|_{L^2}^{(1-s_c)(p-1)}\|Q\|_{\dot{H}^1}^{s_c(p-1)}=\frac{2(p+1)}{3(p-1)}.
\end{equation}
This shows  for $K<0$
\begin{align*}
(1-\delta)M(Q)^{1-s_c}E_0(Q)^{s_c} \geq& M(u)^{1-s_c} E(u)^{s_c}\\
\geq& \|u(t)\|_{L_x^2}^{2(1-s_c)}\Big(\frac12\|u(t)\|_{\dot
H^1}^2-\frac{C_0}{p+1}\|u(t)\|_{L_x^2}^{\frac{5-p}2}
\|u(t)\|_{\dot H^1}^{\frac{3(p-1)}2}\Big)^{s_c}
\end{align*}
for any $t\in I$. This together with
\begin{equation}\label{equ:idenenery}
E_0(Q)=\frac{3p-7}{6(p-1)}\|Q\|_{\dot{H}^1}^2=\frac{3p-7}{4(p+1)}\|Q\|_{L_x^{p+1}}^{p+1},
\end{equation}
implies that
\[
(1-\delta)^\frac1{s_c}\geq
\frac{3(p-1)}{3p-7}\biggl(\frac{\|u(t)\|_{L_x^2}^{1-s_c}\|u(t)\|_{\dot
H^1}^{s_c}}{\|Q\|_{L_x^2}^{1-s_c} \|Q\|_{\dot
H^1}^{s_c}}\biggr)^\frac2{s_c} -
\frac2{3p-7}\biggl(\frac{\|u(t)\|_{L_x^2}^{1-s_c} \|u(t)\|_{\dot
H^1}^{s_c}}{\|Q\|_{L_x^2}^{1-s_c} \|Q\|_{\dot
H^1}^{s_c}}\biggr)^{\frac2{s_c}(p-1)}.
\]
Using a continuity argument, together with the observation that
\[
(1-\delta)^\frac1{s_c} \geq \frac{3(p-1)}{3p-7}y^\frac2{s_c} -
\frac2{3p-7}y^{\frac2{s_c}(p-1)} \Rightarrow |y-1|\geq \delta'
\quad \text{for some}\quad \delta'=\delta'(\delta)>0,
\]
we obtain
\begin{equation}\label{equ:k0negt}
 \|u(t)\|_{L^2}^{1-s_c}\|u(t)\|_{\dot{H}^1}^{s_c}<\|Q\|_{L^2}^{1-s_c}\|Q\|_{\dot{H}^1}^{s_c},\quad \forall~t\in I.
\end{equation}

In sum, we obtain the uniform kinetic energy control in the maximal life-span.
Therefore, we conclude the proof of Theorem \ref{thm:global}.

\begin{remark}\label{rem:tde}
$(i)$ For $K<0$ and $\lambda=-1$, we remark that under the assumption $M(u_0)^{1-s_c}E(u_0)^{s_c}\leq(1-\delta)M(Q)^{1-s_c}E_0(Q)^{s_c}$ for some $\delta>0$, the condition
\begin{equation}\label{equ:weakcond}
  \|u_0\|_{L^2}^{1-s_c}\|u_0\|_{\dot{H}^1}^{s_c}<\|Q\|_{L^2}^{1-s_c}\|Q\|_{\dot{H}^1}^{s_c}
\end{equation}
is equivalent to
\begin{equation}\label{equ:strocond}
  \|u_0\|_{L^2}^{1-s_c}\Big(\|u_0\|_{\dot{H}^1}^2-K\big\||x|^{-\frac12}u_0\big\|_{L^2}^2\Big)^{\frac{s_c}2}<\|Q\|_{L^2}^{1-s_c}\|Q\|_{\dot{H}^1}^{s_c}.
\end{equation}
We take $s_c=\tfrac12$ for example. In this case, we have $p=3,$ and the ground state $Q$ solves
$$-\Delta Q+Q=Q^3.$$
A simple computation shows that
\begin{equation}\label{equ:energ}
  E_0(Q_0)=\frac16\|Q_0\|_{\dot{H}^1}^2=\frac18\|Q_0\|_{L^4}^4=\frac12\|Q_0\|_{L^2}^2
\end{equation}
  and
  \begin{equation}\label{equ:c0}
    C_0:=\frac{\|Q\|_{L^4}^4}{\|Q\|_{L^2}\|Q\|_{\dot{H}^1}^3}=\frac43\frac{1}{\|Q\|_{L^2}\|Q\|_{\dot{H}^1}}.
  \end{equation}

Since $K<0$, it is easy to get \eqref{equ:weakcond} from \eqref{equ:strocond}.

Now, we assume \eqref{equ:weakcond}. By the sharp Gaglilardo-Nirenberg's inequality
$$\|u\|_{L^4}^4\leq C_0\|u\|_{L^2}\|u\|_{\dot{H}^1}^3$$
and using  \eqref{equ:c0}, we obtain
\begin{align*}
  M(u_0)E(u_0)=&\frac12\|u_0\|_{L^2}^2\Big( \|u_0\|_{\dot{H}^1}^2-K\big\||x|^{-\frac12}u_0\big\|_{L^2}^2\Big)-\frac14\|u_0\|_{L^2}^2\|u_0\|_{L^4}^4\\
  \geq&\frac12\|u_0\|_{L^2}^2\Big( \|u_0\|_{\dot{H}^1}^2-K\big\||x|^{-\frac12}u_0\big\|_{L^2}^2\Big)-\frac{C_0}4\|u_0\|_{L^2}^3\|u_0\|_{\dot{H}^1}^3\\
  \geq&\frac12\|u_0\|_{L^2}^2\Big( \|u_0\|_{\dot{H}^1}^2-K\big\||x|^{-\frac12}u_0\big\|_{L^2}^2\Big)-\frac{C_0}4\|Q\|_{L^2}^3\|Q\|_{\dot{H}^1}^3\\
  =&\frac12\|u_0\|_{L^2}^2\Big( \|u_0\|_{\dot{H}^1}^2-K\big\||x|^{-\frac12}u_0\big\|_{L^2}^2\Big)-\frac13\|Q\|_{L^2}^2\|Q\|_{\dot{H}^1}^2.
\end{align*}
This together with the assumption $M(u_0)E(u_0)\leq(1-\delta)M(Q)E_0(Q)$ and \eqref{equ:energ} yields that
\begin{align*}
  \frac12\|u_0\|_{L^2}^2\Big( \|u_0\|_{\dot{H}^1}^2-K\big\||x|^{-\frac12}u_0\big\|_{L^2}^2\Big)\leq & M(u_0)E(u_0)+\frac13\|Q\|_{L^2}^2\|Q\|_{\dot{H}^1}^2\\
  \leq&(1-\delta)M(Q)E_0(Q)+\frac13\|Q\|_{L^2}^2\|Q\|_{\dot{H}^1}^2\\
  =&\frac{3-\delta}{6}\|Q\|_{L^2}^2\|Q\|_{\dot{H}^1}^2.
\end{align*}
And so
$$\|u_0\|_{L^2}^2\Big( \|u_0\|_{\dot{H}^1}^2-K\big\||x|^{-\frac12}u_0\big\|_{L^2}^2\Big)<\|Q\|_{L^2}^2\|Q\|_{\dot{H}^1}^2.$$

$(ii)$ By the same argument as in $(i)$, for $K<0$, $\lambda=-1$ and $p=5$, under the assumption $E(u_0)<E_0(W),$
 the condition
\begin{equation}\label{equ:weakcondener}
  \|u_0\|_{\dot{H}^1}<\|W\|_{\dot{H}^1}
\end{equation}
is equivalent to
\begin{equation}\label{equ:strocondener}
  \|u_0\|_{\dot{H}^1}^2-K\big\||x|^{-\frac12}u_0\big\|_{L^2}^2<\|W\|_{\dot{H}^1}^2.
\end{equation}
\end{remark}

\subsection{Global well-posedness for energy-critical: $s_c=1$ and $K<0$}
We will show the global well-posedness
by controlling global kinetic energy and  proving ``good local
well-posedness" as in Zhang \cite{Zhang}. More precisely, we will show that there exists a
small constant $T=T(\|u_{0}\|_{H^{1}_{x}})$ such that \eqref{equ1.1}
is well-posed on $[0,T]$, which is so-called ``good local
well-posed". On the other hand, since the equation in \eqref{equ1.1}
is time translation invariant, this ``good local well-posed"
combining with the global kinetic energy control gives immediately
the global well-posedness.

{\bf Step 1. global kinetic energy.} For the defocusing case ($\lambda=1$), it follows from Case 1 in Subsection \ref{sub:global} that
$$\sup_{t\in I}\|u(t,\cdot)\|_{H^1}^2\leq C_1M(u_0)+4E(u_0).$$

While for the focusing case $(\lambda=-1)$ and $K<0$, under the restriction
 \begin{equation}\label{equ:energythre123}
   E(u_0)<E_0(W),\; \|u_0\|_{\dot{H}^1}<\|W\|_{\dot{H}^1},
 \end{equation}
 we easily obtain
  \begin{equation}\label{equ:energythre321}
   E_0(u_0)<E_0(W),\; \|u_0\|_{\dot{H}^1}<\|W\|_{\dot{H}^1}.
 \end{equation}
Hence, we have by coercivity as in \cite{KM}
\begin{equation}\label{equ:uniformkneg}
  \sup_{t\in I}\|u(t)\|_{\dot{H}^1}<cE_0(u)<cE(u_0)<cE_0(W).
\end{equation}
And so we derive the  global kinetic energy.

{\bf Step 2: good local well-posedness.} To obtain it,
we first introduce several spaces and give estimates of the nonlinearities in terms of these
spaces. For a time slab $I\subset \R$, we define
$$\dot{X}^0_I:=L_{t,x}^\frac{10}{3}\cap L_t^{10}L_x^\frac{30}{13}(I\times\R^3),\; \dot{X}^1_I:=\{f:\nabla f\in\dot{X}^0_I\},\; X_I^1=\dot{X}^0_I\cap\dot{X}^1_I.$$
Then, we have by H\"older's inequality and Sobolev embedding
\begin{equation}\label{equ:nonlinest}
  \big\|\nabla^i\big(u^kv^{4-k}\big)\big\|_{L_{t,x}^{\frac{10}{7}}(I\times\R^3)}\lesssim
  \|u\|_{\dot{X}^1_x}^{p-1}\|u\|_{\dot{X}^i_I},
\end{equation}
for $i=0,1$, and
\begin{equation}\label{equ:harsob}
  \big\|\langle\nabla\rangle\big(\tfrac{u}{|x|}\big)\big\|_{L_t^2(I;L_x^{\frac{6}{5},2})}\leq C|I|^\frac12\|u\|_{L_t^\infty(I;H^1_x)}.
\end{equation}

Now, it follows from \cite{CKSTT07} for the defocusing case ($\lambda=1$) and \cite{KM} for the focusing case $(\lambda=-1)$ under the assumption \eqref{equ:energythre321} and $u_0$ radial that the Cauchy problem
\begin{align}\label{aequ2}
\begin{cases}
i\partial_tv+\Delta v=\lambda|v|^4v,\quad (t,x)\in \mathbb{R}\times\mathbb{R}^3,\\
v(0)=u_0,
\end{cases}
\end{align}
is globally well-posed and the global solution $v$ satisfies the
estimate
\begin{align}\label{equ31}
\|v\|_{L^{q}(\mathbb{R};\dot{W}^{1,r}_x)}\leqslant
C(\|u_{0}\|_{\dot{H}^{1}}),\quad \|v\|_{L^{q}(\mathbb{R};L^r)}\leq
C(\|u_{0}\|_{\dot H^{1}})\|u_0\|_{L^2}
\end{align}
for all $(q,r)\in\Lambda_0$. So to recover $u$ on the time
interval $[0,T]$, where $T$ is a small constant to be specified
later, it's sufficient to solve the difference equation of $\omega$
with 0-data initial on the time interval $[0,T]$,
\begin{align}\label{equ32}
\begin{cases}
i\omega_t+\Delta \omega=-\tfrac{K}{|x|}(v+\omega)-\lambda|v+\omega|^4(v+\omega)+\lambda|v|^4v\\
\omega(0)=0.
\end{cases}
\end{align}
In order to solve \eqref{equ32}, we subdivide $[0,T]$ into finite
subintervals such that on each subinterval, the influence of $v$ to
the problem \eqref{equ32} is very small.

Let $\epsilon$ be a small constant, from \eqref{equ31}, it allows us
to divide $\mathbb{R}$ into subintervals $I_{0},\ldots I_{J-1}$ such
that on each $I_{j}$,
\begin{align*}
\|v\|_{X^1(I_{j})}\thicksim \epsilon,\quad 0\leq j\leq J-1\quad
\text{with}~J\leq C(\|u_{0}\|_{H^1},\epsilon).
\end{align*}
So without loss of generality and renaming the intervals if
necessary, we can write
\begin{align*}
[0,T]=\bigcup\limits_{j=0}^{J^\prime}I_{j},\quad
I_{j}=[t_{j},t_{j+1}]
\end{align*}
with $J^{\prime}\leqslant J$ and on each
$I_{j}$\begin{equation}\label{ad1}\|v\|_{X^1(I_j)}\lesssim\epsilon.\end{equation}
Now we begin to solve the difference equation \eqref{equ32} on each
$I_{j}$ by inductive arguments. More precisely, we show that for
each $0\leqslant j\leqslant J^\prime-1$, there exists a unique
solution $\omega$ to \eqref{equ32} on $I_{j}$ such that
\begin{align}\label{equ33}
\|\omega\|_{X^1(I_j)}+\|\omega\|_{L^{\infty}(I_{j};H^{1})}\leq
(2C)^{j}T^{\frac{1}{4}}.
\end{align}
 We mainly utilize the induction argument. Assume \eqref{equ32} has been solved on $I_{j-1}$ and the solution
$\omega$ satisfies the bound \eqref{equ33} until to $j-1$, it is enough
to derive the bound of the $\omega$ on $I_{j}$.

Define the solution map
\begin{align*}
\Phi(\omega(t))=e^{i(t-t_{j})\Delta}w(t_j)+i\int_{t_j}^{t}e^{i(t-s)\Delta}
\Big(\tfrac{K}{|x|}(v+\omega)+\lambda|v+\omega|^4(v+\omega)-\lambda|v|^4v\Big)(s)ds
\end{align*}
and a set
\begin{align*}
B=\{\omega:\|\omega\|_{L^\infty(I_{j};H^1)}+\|\omega\|_{X^1(I_j)}\leq
(2C)^jT^{\frac{1}{4}}\}
\end{align*}
and the norm $\|\cdot\|_B$ is taken as the same as the one in the
capital bracket. Then it suffices to show that $B$ is stable and the
solution map $\Phi$ is contractive under the weak topology
$\dot{X}^0(I_j)\cap L_t^\infty(I_j,L_x^2)$. Actually, it follows from the Strichartz estimate on Lorentz space and \eqref{equ:nonlinest}, \eqref{equ:harsob} that
\begin{align*}
\|\Phi(\omega)\|_{B}\lesssim&
\|\omega(t_j)\|_{H^1}+\big\||v+w|^4(v+w)-|v|^4v\big\|_{L^{\frac{10}{7}}(I_j;W^{1,\frac{10}{7}}_x)}
+\Big\|\langle \nabla\rangle \Big(\tfrac{K}{|x|}(v+\omega)\Big)\Big\|_{L^2(I_j;L^{\frac{6}{5},2})}\\
\lesssim&\|\omega(t_j)\|_{H^1}+\sum\limits_{i=0}^{4}\|v\|_{X^1(I_j)}^i
\|\omega\|_{X^1(I_j)}^{5-i}+T^{\frac{1}{2}}\|v+\omega\|_{L_t^\infty(I_j; H^1_x)}
\end{align*}
Thus, \eqref{equ31} and \eqref{ad1} gives
\begin{align*}
\|\Phi(\omega)\|_{B}&\leq C\Big(
\|\omega(t_j)\|_{H^1}+\sum\limits_{i=0}^{4}\epsilon^i\|\omega\|_{X^1(I_j)}^{5-i}+CT^{\frac{1}{2}}+T^{\frac{1}{2}}
\|\omega\|_{L_t^\infty (I_j;H^1_x)}\Big).
\end{align*}
Plugging the inductive assumption $\|\omega(t_j)\|_{H^1}\leq
(2C)^{j-1}T^{\frac{1}{4}}$, we see that for $\omega\in B$,
\begin{align}\label{equ36}
\|\Phi(\omega)\|_{B}
\leq& C\big[(2C)^{j-1}
+\epsilon^4(2C)^j+CT^\frac14+(2C)^{j}T^{\frac{3}{4}}\big]T^{\frac{1}{4}}\\\label{equ37}
&+C\sum\limits_{i=0}^3((2C)^jT^{\frac{1}{4}})^{5-i}\epsilon^i
\end{align}
Thus we can choose $\epsilon$ and $T$ small depending only on the Strichartz
constant such that
\begin{align*}
\eqref{equ36}\leq\frac{3}{4}(2C)^jT^{\frac{1}{4}}.
\end{align*}
Fix this $\epsilon$, \eqref{equ37} is a higher order term with
respect to the quantity $T^{\frac{1}{4}}$, we have
\begin{align*}
\eqref{equ37}\leq\frac{1}{4}(2C)^jT^{\frac{1}{4}},
\end{align*}
which is available by choosing $T$ small enough. Of course $T$ will
depend on $j$, however, since $j\leqslant J^\prime-1\leq
C(\|u_{0}\|_{H^1})$, we can choose $T$ to be a small constant
depending only on $\|u_{0}\|_{H^1}$ and $\epsilon$, therefore is
uniform in the process of induction. Hence
\begin{align*}
\|\Phi(\omega)\|_{B}&\leq (2C)^jT^{\frac{1}{4}}.
\end{align*}
On the other hand, by a similarly argument as before, we have, for
$\omega_1, \omega_2\in B$
\begin{align*}
&\|\Phi(\omega_{1})-\Phi(\omega_2)\|_{\dot{X}^0{(I_j)}\cap L^\infty_t(I_j,L_x^2)} \\ \leq&C\big\|\tfrac{\omega_1-\omega_2}{|x|}\big\|_{L_t^2(I_j;L_x^{\frac{6}{5}})}+C\big\||v+\omega_1|^4(v+\omega_1)-|v+\omega_2|^4(v+\omega_2)\big\|_{
L_{t,x}^\frac{10}{7}(I_j\times\R^3)}\\
\leq&CT^\frac12\|\omega_1-\omega_2\|_{L^\infty_t(I_j,L_x^2)}+C\|\omega_1-\omega_2\|_{\dot{X}^0{(I_j)}}\big(
\|v\|_{\dot{X}^1(I_j)}^4+\|\omega_1\|_{\dot{X}^1(I_j)}^4+\|\omega_2\|_{\dot{X}^1(I_j)}^4\big)\\
\leq&\|\omega_1-\omega_2\|_{\dot{X}^0{(I_j)}\cap L^\infty_t(I_j,L_x^2)}\big(CT^\frac12+\epsilon^4+2(2C)^jT^{\frac{1}{4}}\big),
\end{align*}
which allows us to derive
\begin{equation*}
\|\Phi(\omega_{1})-\Phi(\omega_2)\|_{\dot{X}^0{(I_j)}\cap L^\infty_t(I_j,L_x^2)}\leq\frac{1}{2}\|\omega_1-\omega_2\|_{\dot{X}^0{(I_j)}\cap L^\infty_t(I_j,L_x^2)},
\end{equation*}
by taking $\epsilon,T$ small such that
\begin{equation*}
CT^\frac12+\epsilon^4+2(2C)^jT^{\frac{1}{2}}\leq\frac{1}{4}.
\end{equation*}
A standard fixed point argument gives a unique solution $\omega$ of
\eqref{equ32} on $I_j$ which satisfies the bound \eqref{equ33}.
Finally, we get a unique solution of \eqref{equ32} on $[0,T]$ such
that
\begin{align*}
\|\omega\|_{X^1([0,T])}\leq\sum\limits_{j=0}^{J^\prime-1}\|\omega\|_{X^1(I_j)}
\leq\sum\limits_{j=0}^{J^\prime-1}(2C)^jT^{\frac{1}{4}}\leq
C(2C)^JT^{\frac{1}{2}}\leq C.
\end{align*}
Since on $[0,T]$, $u=v+\omega$, we obtain a unique solution to
\eqref{equ1.1} on $[0,T]$ such that
\begin{align*}
\|u\|_{X^1([0,T])}\leq
\|\omega\|_{X^1([0,T])}+\|v\|_{X^1([0,T])}\leq C(\|u_{0}\|_{H^1}).
\end{align*}

As we mentioned before, this ``good local well-posedness" combining
with the ``global kinetic energy control" as in Step 1 gives finally the global
well-posedness. However, since the solution is connected one
interval by another, it does not  have global space-time bound. In
the following, we will discuss the defocusing case, in which the global
solution have the enough decay to imply scattering.



\section{Morawetz estimate and scattering theory}
In this section, we establish an interaction Morawetz estimate and the scattering theory in Theorem \ref{thm:scattering}.
In the whole of the section, we are in the defocusing case with repulsive potential, that is,  $K<0$ and $\lambda=1$.

\subsection{Morawetz estimate} In this subsection, we establish the interaction Morawetz estimate for  \eqref{equ1.1} with  $K<0$ and $\lambda=1$.
\begin{lemma}\label{lem:virial}
Let  $u:\R\times\R^3\to \C$ solve $i\pa_tu+\Delta u+V(x) u=\mathcal{N},$ and $\mathcal{N}\bar{u}\in \R$. Given a smooth weight $w:\R^3\to\R$ and a (sufficiently smooth and
decaying) solution $u$ to \eqref{equ1.1}, we define
$$I(t,w)=\int_{\R^3} w(x)|u(t,x)|^2\;dx.$$
Then, we have
\begin{align}\label{equ:firstd}
\pa_tI(t,w)=&2{\rm Im}\int_{\R^3}\bar{u}\nabla u\cdot \nabla w\;dx,\\ \label{equ:secdd}
\pa_{tt}I(t,w)=&-\int_{\R^3}|u|^2\Delta^2w\;dx+4{\rm Re}\int \pa_ju\pa_k\bar{u}\pa_j\pa_kw\\\nonumber
&+\int_{\R^3}|u|^2\nabla V\cdot\nabla w\;dx
+2{\rm Re}\int\big(\mathcal{N}\nabla\bar{u}-\bar{u}\nabla\mathcal{N}\big)\cdot\nabla w\;dx.
\end{align}

\end{lemma}

\begin{proof}
First, note that
\begin{equation}\label{equ:ut}
\pa_tu=i\Delta u+iV(x)u-i\mathcal{N},
\end{equation}
we get
\begin{align*}
\pa_tI(t,w)=&2{\rm Re}\int_{\R^3}w(x)\pa_tu\bar{u}\;dx\\
=&2{\rm Re}\int_{\R^3}w(x)\big(i\Delta u+iV(x)u-i\mathcal{N}\big)\bar{u}\;dx\\
=&-2{\rm Im}\int_{\R^3}w(x)\Delta u\bar{u}\;dx\\
=&2{\rm Im}\int_{\R^3}\bar{u}\nabla u\cdot \nabla w\;dx.
\end{align*}

Furthermore,
\begin{align*}
\pa_{tt}I(t,w)=&2{\rm Im}\int_{\R^3}\bar{u}_t\nabla u\cdot \nabla w\;dx+2{\rm Im}\int_{\R^d}\bar{u}\nabla u_t\cdot \nabla w\;dx\\
=&2{\rm Im}\int_{\R^3}\big(-i\Delta \bar{u}-iV(x)\bar{u}+i\bar{\mathcal{N}}\big)\nabla u\cdot \nabla w\;dx\\
&+2{\rm Im}\int_{\R^3}\bar{u}\nabla \big(i\Delta u+iV(x)u-i\mathcal{N}\big)\cdot \nabla w\;dx\\
=&2{\rm Re}\int \big(-\Delta \bar{u}\nabla u+\bar{u}\nabla\Delta u\big)\cdot\nabla w\;dx\\
&+2{\rm Re}\int\big(\bar{u}\nabla(Vu)-V\bar{u}\nabla u\big)\cdot\nabla w\;dx\\
&+2{\rm Re}\int\big(\mathcal{N}\nabla\bar{u}-\bar{u}\nabla\mathcal{N}\big)\cdot\nabla w\;dx\\
=&-\int_{\R^3}|u|^2\Delta^2w\;dx+4{\rm Re}\int \pa_ju\pa_k\bar{u}\pa_j\pa_kw\\
&+\int_{\R^3}|u|^2\nabla V\cdot\nabla w\;dx
+2{\rm Re}\int\big(\mathcal{N}\nabla\bar{u}-\bar{u}\nabla\mathcal{N}\big)\cdot\nabla w\;dx.
\end{align*}

\end{proof}

\begin{remark}\label{rem:moride}
$(i)$ For $\mathcal{N}=\lambda|u|^{p-1}u$, so $\mathcal{N}\bar{u}=\lambda |u|^{p+1}\in\R$, then one has
$$2{\rm Re}\int_{\R^3}\big(\mathcal{N}\nabla\bar{u}-\bar{u}\nabla\mathcal{N}\big)\cdot\nabla w\;dx=\lambda\frac{2(p-1)}{p+1}\int_{\R^d}|u|^{p+1}\Delta w\;dx.$$

$(ii)$ For  $\mathcal{N}=|u|^{p-1}u,\;V(x)=\frac{K}{|x|}$, and $w$ being radial, we have
\begin{align*}
\pa_{tt}I(t,w)=&-\int_{\R^3}|u|^2\Delta^2w\;dx+4{\rm Re}\int \pa_ju\pa_k\bar{u}\pa_j\pa_kw\\\nonumber
&-K\int_{\R^3}\frac{|u|^2}{|x|^2}\pa_r w\;dx
+\frac{2(p-1)}{p+1}\int_{\R^3}|u|^{p+1}\Delta w\;dx.
\end{align*}

\end{remark}

As a consequence, we obtain the following classical Morawetz estimate by taking $w(x)=|x|$.
\begin{lemma}[Classical Morawetz estimate]\label{lem:moraw}
Let $u:\;I\times\R^3\to\mathbb{C}$ solve \eqref{equ1.1} with $\lambda=1$. Then,
\begin{align}\label{equ:morident}
\frac{d}{dt}{\rm Im}\int_{\R^3}\bar{u}\frac{x}{|x|}\cdot\nabla u\;dx=&c|u(t,0)|^2+2\int_{\R^3}\frac{|\nabla_\theta u|^2}{|x|}\;dx\\\nonumber
&-\frac{K}{2}\int_{\R^3}\frac{|u|^2}{|x|^2}\;dx+\int_{\R^3}\frac{|u|^4}{|x|}\;dx.
\end{align}
Moreover, we have for $K<0$
\begin{equation}\label{equ:Morawe}
\int_I\int_{\R^3}\Big(\frac{|u|^2}{|x|^2}+\frac{|u|^4}{|x|}\Big)\;dx\;dt\leq C\sup_{t\in I}\|u(t,\cdot)\|_{\dot H^\frac12}^2.
\end{equation}
\end{lemma}

Next, we establish the interaction Morawetz estimate for \eqref{equ1.1} with $K<0$ and $\lambda=1$ as the case that $K=0$ in \cite{CKSTT}.

\begin{theorem}[Interaction Morawetz estimate]\label{thm:intmorawet}
Let $u:\;I\times\R^3\to\C$ solve $i\pa_tu+\Delta u+\frac{K}{|x|}u=|u|^{p-1}u.$ Then, for $K<0$, we have
\begin{equation}\label{interac2}
\big\|u\big\|_{L_t^4(I;L_x^4(\R^3))}^2\leq
C\|u(t_0)\|_{L^2}\sup_{t\in I}\|u(t)\|_{\dot H^{\frac12}}.
\end{equation}

\end{theorem}

\begin{proof}

We consider the NLS equation in the form of
\begin{equation}\label{NLS}
i\partial_tu+\Delta u=gu
\end{equation}
where $g=g(\rho,|x|)$ is a real function of $\rho=|u|^2=2T_{00}$ and
$|x|$.  We first recall the conservation laws for free Schr\"odinger
in Tao \cite{Tao}
\begin{equation*}
\begin{split}
\partial_t T_{00}+\partial_j T_{0j}=0,\\
\partial_t T_{0j}+\partial_k T_{jk}=0,
\end{split}
\end{equation*}
where the mass density quantity $T_{00}$ is defined by
$T_{00}=\tfrac12|u|^2,$ the mass current and the momentum density
quantity $T_{0j}=T_{j0}$ is given by $T_{0j}=T_{j0}=\mathrm{Im}(\bar
u\partial_j u)$, and the quantity $T_{jk}$ is
\begin{equation}\label{stress}
T_{jk}=2\mathrm{Re}(\partial_j u
\partial_k\bar u)-\tfrac12\delta_{jk}\Delta(|u|^2),
\end{equation}
for all $j,k=1,...n,$ and $\delta_{jk}$ is the Kroncker delta. Note
that the kinetic terms are unchanged, we see that for \eqref{NLS}
\begin{equation}\label{Local Conservation}
\begin{split}
\partial_t T_{00}+\partial_j T_{0j}&=0,\\
\partial_t T_{0j}+\partial_k T_{jk}&=-\rho\partial_j g.
\end{split}
\end{equation}
By the density argument, we may assume sufficient smoothness and
decay at infinity of the solutions to the calculation and in
particular to the integrations by parts. Let $h$ be a sufficiently
regular real even function defined in $\R^3$, e.g. $h=|x|$. The
starting point is the auxiliary quantity
\begin{equation*}
J=\tfrac12\langle|u|^2, h\ast |u|^2\rangle=2\langle T_{00}, h\ast
T_{00}\rangle.
\end{equation*}
Define the quadratic Morawetz quantity $M=\tfrac14\partial_t J$.
Hence we can precisely rewrite
\begin{equation}\label{3.1}
M=-\tfrac12\langle\partial_jT_{0j}, h\ast
T_{00}\rangle-\tfrac12\langle T_{00}, h\ast
\partial_jT_{0j} \rangle=-\langle T_{00}, \partial_j h\ast
T_{0j} \rangle.
\end{equation}
By \eqref{Local Conservation} and integration by parts, we have
\begin{equation*}
\begin{split}
\partial_tM&=\langle\partial_kT_{0k}, \partial_j h\ast T_{0j} \rangle-\langle T_{00},
\partial_j h\ast\partial_t T_{0j} \rangle\\&=-\sum_{j,k=1}^n\langle T_{0j}, \partial_{jk} h\ast T_{0j} \rangle+\langle T_{00},
\partial_{jk} h\ast T_{jk} \rangle+\langle \rho,
\partial_j h\ast(\rho\partial_j g) \rangle.
\end{split}
\end{equation*}
For our purpose, we note that
\begin{equation}
\begin{split}
\sum_{j,k=1}^n\langle T_{0k},  \partial_{jk} h\ast T_{0j}
\rangle&=\big\langle \mathrm{Im}(\bar u\nabla u), \nabla^2 h\ast
\mathrm{Im}(\bar u\nabla u) \big\rangle\\&=\big\langle \bar u\nabla
u, \nabla^2 h\ast \bar u\nabla u \rangle-\langle \mathrm{Re}(\bar
u\nabla u), \nabla^2 h\ast \mathrm{Re}(\bar u\nabla u) \big\rangle.
\end{split}
\end{equation}
Therefore it yields that
\begin{equation*}
\begin{split}
\partial_tM=&\big\langle \mathrm{Re}(\bar u\nabla
u), \nabla^2 h\ast \mathrm{Re}(\bar u\nabla u)
\big\rangle-\big\langle \bar u\nabla u, \nabla^2 h\ast \bar u\nabla
u \big\rangle\\&+\Big\langle \bar uu,
\partial_{jk} h\ast \big(\mathrm{Re}(\partial_j u \partial_k\bar
u)-\tfrac14\delta_{jk}\Delta(|u|^2)\big) \Big\rangle+\big\langle
\rho,
\partial_j h\ast(\rho\partial_j g) \big\rangle.
\end{split}
\end{equation*}
From the observation
\begin{equation*}
\begin{split}
-\big\langle \bar uu,
\partial_{jk} h\ast\delta_{jk}\Delta(|u|^2) \big\rangle=\big\langle \nabla (|u|^2), \Delta h\ast
\nabla(|u|^2) \big\rangle,
\end{split}
\end{equation*}
we write
\begin{equation}\label{Morawetz equality}
\begin{split}
\partial_tM=\tfrac12\langle \nabla \rho, \Delta h\ast\nabla\rho \rangle+R+\big\langle \rho,
\partial_j h\ast(\rho\partial_j g) \big\rangle,
\end{split}
\end{equation}
where $R$ is given by
\begin{equation*}\label{3.4}
\begin{split}
R&=\big\langle \bar uu, \nabla^2 h\ast (\nabla\bar u \nabla u)
\big\rangle-\big\langle \bar u\nabla u, \nabla^2 h\ast \bar u\nabla
u \big\rangle\\&=\tfrac12\int \Big(\bar u(x)\nabla \bar u(y)-\bar
u(y)\nabla\bar u(x)\Big)\nabla^2h(x-y)\Big(u(x)\nabla
u(y)-u(y)\nabla u(x)\Big)\mathrm{d}x\mathrm{d}y.
\end{split}
\end{equation*}

Since the Hessian of $h$ is positive definite, we have $R\geq0$. Integrating
over time in an interval $[t_1, t_2]\subset I$ yields
\begin{equation*}
\begin{split}
\int_{t_1}^{t_2}\Big\{\frac12\langle \nabla \rho, \Delta
h\ast\nabla\rho \rangle+\langle \rho,
\partial_j h\ast(\rho\partial_j g) \rangle+R\Big\}\mathrm{d}t=-\langle T_{00}, \partial_j h\ast
T_{0j} \rangle\big|_{t=t_1}^{t=t_2}.
\end{split}
\end{equation*}

From now on, we choose $h(x)=|x|$. One can follow the arguments in
\cite{CKSTT} to bound the right hand by the quantity
\begin{equation*}
\Big|\mathrm{Im}\iint_{\R^{3}\times\R^3}|u(x)|^2\frac{x-y}{|x-y|}\bar
u(y)\nabla u(y)dxdy\Big|\leq C\sup_{t\in
I}\|u(t)\|^2_{L^2}\|u(t)\|^2_{\dot H^{\frac12}}.
\end{equation*}
Therefore we conclude
\begin{equation}\label{Morawetz inequality}
\int_{t_1}^{t_2}\big\langle \rho,
\partial_j h\ast(\rho\partial_j g) \big\rangle dt+\big\|u\big\|_{L^4(I;L^4(\R^3))}^2\leq C\sup_{t\in
I}\|u(t)\|_{L^2}\|u(t)\|_{\dot H^{\frac12}}.
\end{equation}

Now we consider the term
\begin{equation*}
\begin{split}
P&:=\big\langle \rho, \nabla h\ast (\rho\nabla g) \big\rangle.
\end{split}
\end{equation*}
Consider $g(\rho,|x|)=\rho^{(p-1)/2}+V(x)$, then we can write
$P=P_1+P_2$ where
\begin{equation}\label{3.5}
\begin{split}
P_1= \big\langle\rho, \nabla h\ast \big(\rho\nabla
(\rho^{(p-1)/2})\big)\big\rangle=\frac{p-1}{p+1}\big\langle\rho,
\Delta h\ast \rho^{(p+1)/2}\big\rangle\geq 0
\end{split}
\end{equation}
and
\begin{equation}\label{P2}
\begin{split}
P_2=& \iint\rho(x)\nabla h(x-y)\rho(y)\nabla
\big(V(y)\big)\mathrm{d}x\mathrm{d}y\\
=&\iint|u(x)|^2\frac{(x-y)\cdot y}{|x-y|\cdot|y|^3}|u(y)|^2\;dx\;dy.
\end{split}
\end{equation}
By using the Morawetz estimate \eqref{equ:Morawe}
$$\int_I\int_{\R^3}\frac{|u|^2}{|x|^2}\;dx\;dt\leq C\sup_{t\in I}\|u\|_{\dot{H}^\frac12}^2,$$
one has
$$|P_2|\leq \|u_0\|_{L^2}^2\sup_{t\in I}\|u\|_{\dot{H}^\frac12}^2.$$
And so, we conclude the proof of Theorem \ref{thm:intmorawet}.

\end{proof}

\begin{remark}\label{rem:intmor}
By the same argument as above, one can extend the Coulomb potential $V(x)=\frac{K}{|x|}$ to $V(x)$ satisfies the following argument: first, we have by Morawetz estimate
$$\int_I\int_{\R^3}|u|^2\frac{x}{|x|}\cdot\nabla V\;dx\;dt\leq C\sup_{t\in I}\|u\|_{\dot{H}^\frac12}^2.$$
As in \eqref{P2}, we are reduced to estimate the term
$$\int_I\int_{\R^3}|u|^2|\nabla V|\;dx\;dt.$$
Therefore, we can extend  $V(x)$ satisfying
$$\frac{x}{|x|}\cdot \nabla V\geq c|\nabla V|,$$
with some positive constant $c$.

\end{remark}

\subsection{Scattering theory}

 Now we use the global-in-time
interaction Morawetz estimate \eqref{interac2}
\begin{equation}\label{rMorawetz}
\big\|u\big\|_{L_t^4(\R;L_x^4(\R^3))}^2\leq
C\|u_0\|_{L^2}\sup_{t\in \R}\|u(t)\|_{\dot H^{\frac12}},
\end{equation}
to prove the scattering theory part of Theorem \ref{thm:scattering}. Since the construction of the wave operator is
standard, we only show the asymptotic completeness.\vspace{0.2cm}

Let $u$ be a global solution to \eqref{equ1.1}. Let $\eta>0$ be a small constant to be chosen later and split
$\R$ into $L=L(\|u_0\|_{H^1})$ finite subintervals $I_j=[t_j,t_{j+1}]$ such
that
\begin{equation}\label{equ4.16}
\|u\|_{L_{t,x}^{4}(I_j\times\R^3)}\leq\eta.
\end{equation}

 Define
$$\big\|\langle\nabla\rangle u\big\|_{S^0(I)}:=\sup_{(q,r)\in\Lambda_0:r\in[2,3_-]}\big\|\langle\nabla\rangle u\big\|_{L_t^qL_x^r(I\times\R^3)}.$$
Using the Strichartz estimate and Sobolev norm equivalence \eqref{equ:sobequi}, we obtain
\begin{align}\label{star}
\big\|\langle\nabla\rangle
u\big\|_{S^0(I_j)}\lesssim&\|u(t_j)\|_{H^1}+\big\|\langle\nabla\rangle(|u|^{p-1}u)
\big\|_{L_t^2L_x^\frac{6}{5}(I_j\times\R^3)}.
\end{align}
Let $\epsilon>0$ to be determined later, and
$r_\epsilon=\frac{6}{3-(4/(2+\epsilon))}$. On the other hand, we
use the Leibniz rule and H\"older's inequality to obtain
\begin{align*}
\big\|\langle\nabla\rangle(|u|^{p-1}u)
\big\|_{L_t^2L_x^\frac65}\lesssim&\big\|\langle\nabla\rangle
u\big\|_{L_t^{2+\epsilon}(I_j;L_x^{r_\epsilon})}
\|u\|^{p-1}_{L_t^{\frac{2(p-1)(2+\epsilon)}{\epsilon}}L_x^\frac{3(p-1)(2+\epsilon)}{4+\epsilon}}.\end{align*}

Taking $\epsilon=2_+$, and so $r_\epsilon=3_-$. If
$p\in(\frac73,4]$, then ${2(p-1)(2+\epsilon)}/{\epsilon}>4$ and
$2\leq \frac{3(p-1)(2+\epsilon)}{4+\epsilon}\leq 6$.
 Therefore we use interpolation to
obtain
\begin{align*}
\|u\|_{L_t^{\frac{2(p-1)(2+\epsilon)}{\epsilon}}L_x^\frac{3(p-1)(2+\epsilon)}{4+\epsilon}}\leq C\|u\|^\alpha_{L_{t,x}^{4}(I_j\times\R^3)}\|u\|^\beta_{L_t^{\infty}L_x^6(I_j\times\R^3)}\|u\|^\gamma_{L_t^{\infty}L_x^2(I_j\times\R^3)},
\end{align*}
where $\alpha>0,\beta,\gamma\geq 0$ satisfy $\alpha+\beta+\gamma=1$ and
\begin{align*}
\begin{cases}
\frac{\epsilon}{2(p-1)(2+\epsilon)}&=\frac{\alpha}{4}+\frac{\beta}{\infty}+\frac{\gamma}{\infty},\\
\frac{4+\epsilon}{3(p-1)(2+\epsilon)}&=\frac{\alpha}{4}
+\frac{\beta}{6}+\frac{\gamma}2.
\end{cases}
\end{align*}
Hence
\begin{align*}
\big\|\langle\nabla\rangle(|u|^{p-1}u)
\big\|_{L_t^2L_x^\frac{2n}{n+2}}\lesssim&\big\|\langle\nabla\rangle
u\big\|_{L_t^{2+\epsilon}(I_j;L_x^{r_\epsilon})}\|u\|^{\alpha(p-1)}_{L_{t,x}^{4}(\R\times\R^3)}
\|u\|^{(\beta+\gamma)(p-1)}_{L_t^\infty H^1_x(I_j\times\R^3)}\\\leq&
C\eta^{\alpha(p-1)}\big\|\langle\nabla\rangle
u\big\|_{S^0(I_j)}.\end{align*} Plugging this into \eqref{star} and
noting that $\alpha(p-1)>0$, we can choose $\eta$ to be small enough
such that
\begin{align*}
\big\|\langle\nabla\rangle
u\big\|_{S^0(I_j)}\leq C(E,M,\eta).\end{align*}
Hence we have by the finiteness of $L$
\begin{align}\label{boundnorm}
\big\|\langle\nabla\rangle
u\big\|_{S^0(\R)}\leq C(E,M,\eta,L).\end{align}

 If
$p\in (4,5)$, we use interpolation to show that
\begin{align*}
\|u\|_{L_t^{\frac{2(p-1)(2+\epsilon)}{\epsilon}}L_x^\frac{3(p-1)(2+\epsilon)}{4+\epsilon}}\leq C\|u\|^\alpha_{L_{t,x}^{4}(I_j\times\R^3)}\|u\|^\beta_{L_t^{\infty}L_x^6(I_j\times\R^3)}\|u\|^\gamma_{L_t^{6}L_x^{18}(I_j\times\R^3)},\end{align*}
where  $\alpha>0,\beta,\gamma\geq 0$ satisfy $\alpha+\beta+\gamma=1$ and
\begin{align*}
\begin{cases}
\frac{\epsilon}{2(p-1)(2+\epsilon)}&=\frac{\alpha}{4}+\frac{\beta}{\infty}+\frac{\gamma}{6},\\
\frac{4+\epsilon}{3(p-1)(2+\epsilon)}&=\frac{\alpha}{4}
+\frac{\beta}{6}+\frac{\gamma}{18}.
\end{cases}
\end{align*} It is easy to solve
these equations for $p\in(4,5)$. Since $r_\epsilon\in[2,3_-]$ for
$\epsilon=2_+$, we have
\begin{align*}
\big\|\langle\nabla\rangle(|u|^{p-1}u)
\big\|_{L_t^2L_x^\frac{6}{5}}\lesssim&\big\|\langle\nabla\rangle
u\big\|_{L_t^{2+\epsilon}(I_j;L_x^{r_\epsilon})}\|u\|^{\alpha(p-1)}_{L_{t,x}^4(\R\times\R^3)}\|u\|^{\beta(p-1)}_{L_t^\infty
H^1_x(I_j\times\R^3)}\|\langle\nabla\rangle u\|^{\gamma(p-1)}_{L_t^{6}L_x^{\frac{18}{7}}(I_j\times\R^3)}\\&\leq C\eta^{\alpha(p-1)}\big\|\langle\nabla\rangle
u\big\|^{1+\gamma(p-1)}_{S^0(I_j)}.\end{align*}
Hence arguing as above we obtain \eqref{boundnorm}.

Finally, we utilize \eqref{boundnorm} to show asymptotic completeness. We need to prove that there exist unique $u_\pm$ such that
$$\lim_{t\to\pm\infty}\|u(t)-e^{it\lk}u_\pm\|_{H^1_x}=0.$$
By time reversal symmetry, it suffices to prove this for positive
times. For $t>0$, we will show that $v(t):=e^{-it\lk}u(t)$ converges
in $H^1_x$ as $t\to+\infty$, and denote $u_+$ to be the limit. In
fact, we obtain by Duhamel's formula
\begin{equation}\label{equ4.21}
v(t)=u_0-i\int_0^te^{-i\tau \lk}(|u|^{p-1}u)(\tau)d\tau.
\end{equation}
Hence, for $0<t_1<t_2$, we have
$$v(t_2)-v(t_1)=-i\int_{t_1}^{t_2}e^{-i\tau \lk}(|u|^{p-1}u)(\tau)d\tau.$$
Arguing as before, we deduce that for some $\alpha>0,\beta\geq1$
\begin{align*}
\|v(t_2)-v(t_1)\|_{H^1(\R^3)}=&\Big\|\int_{t_1}^{t_2}e^{-i\tau \lk}(|u|^{p-1}u)(\tau)d\tau\Big\|_{H^1(\R^3)}\\
\lesssim&\big\|\langle\nabla\rangle(|u|^{p-1}u)
\big\|_{L_t^2L_x^\frac65([t_1,t_2]\times\R^3)}\\
\lesssim&\|u\|_{L_{t,x}^{4}([t_1,t_2]\times\R^3)}^{\alpha(p-1)}\big\|\langle\nabla\rangle
u\big\|^\beta_{S^0([t_1,t_2])}
\\
\to&0\quad \text{as}\quad t_1,~t_2\to+\infty.
\end{align*}
As $t$ tends to $+\infty$, the limitation of \eqref{equ4.21} is well
defined. In particular, we find the asymptotic state
$$u_+=u_0-i\int_0^\infty e^{-i\tau \lk}(|u|^{p-1}u)(\tau)d\tau.$$
Therefore, we conclude the proof of Theorem \ref{thm:scattering}.



\section{Blow up}

In this section, we study the blow up behavior of the solution in the focusing case, i.e $\lambda=-1$. In the case that $K>0$, we will use the sharp Hardy's inequality and Young's inequality to obtain
\begin{align}\nonumber
 \int_{\R^3}\frac{|u|^2}{|x|}\;dx\leq &\big(\int|u|^2\;dx\big)^\frac12\Big(\int\frac{|u|^2}{|x|^2}\;dx\Big)^\frac12 \\\nonumber
  \leq & 2\|u\|_{L^2}\|u\|_{\dot{H}^1}\\\label{equ:kpos}
  \leq&\frac{1}{C_{p,K}}\|u\|_{L^2}^2+C_{p,K}\|u\|^2_{\dot{H}^1},
\end{align}
for any $C_{p,K}>0$.

From Remark \ref{rem:moride}, it follows that
for $w$  radial function, we have
\begin{align}\label{equ:virit}
\pa_{tt}\int_{\R^3}w(x)|u|^2\;dx=&-\int_{\R^3}|u|^2\Delta^2w\;dx+4{\rm Re}\int_{\R^3} \pa_ju\pa_k\bar{u}\pa_j\pa_kw\\\nonumber
&-K\int_{\R^3}\frac{|u|^2}{|x|^2}\pa_r w\;dx
-\frac{2(p-1)}{p+1}\int_{\R^3}|u|^{p+1}\Delta w\;dx.
\end{align}

{\bf Case 1: $u_0\in\Sigma.$}
By taking $w(x)=|x|^2$, we obtain
\begin{corollary}\label{cor:virial}
Let $u$ solve \eqref{equ1.1}, then we have
\begin{align}\nonumber
\frac{d^2}{dt^2}\int_{\R^3}|x|^2|u(t,x)|^2\;dx=&8\int_{\R^3}|\nabla u|^2\;dx-2K\int_{\R^3}\frac{|u|^2}{|x|}\;dx-\frac{12(p-1)}{p+1}\int_{\R^3}|u|^{p+1}\;dx\\\nonumber
=&12(p-1)E(u)-2(3p-7)\int_{\R^3}|\nabla u|^2\;dx+6K\int_{\R^3}\frac{|u|^2}{|x|}\;dx.
\end{align}
\end{corollary}

Let $I=[0,T]$ be the maximal interval of existence.
Let
$$y(t):=\int|x|^2|u(t,x)|^2\;dx,$$ then for $t\in I$
$$y'(t)=4{\rm Im}\int_{\R^3}x\cdot\nabla u\bar{u}\;dx.$$
By Corollary \ref{cor:virial} and \eqref{equ:kpos} with $C_{p,K}=\frac{3p-7}{3K}$ when $K>0$, we get
\begin{equation}\label{equ:y''t}
  y''(t)\leq 12(p-1)C(E,M):= \begin{cases}
  12(p-1)E(u_0) \quad\text{if}\quad K\leq0\\
  12(p-1)E(u_0)+\frac{18K^2}{3p-7}M(u_0)\quad\text{if}\quad K>0.
  \end{cases}
\end{equation}
Hence
$$y(t)\leq 6(p-1)C(E,M)t^2+y'(0)t+y(0).$$
which implies $I$ is finite provided that
$$(i) C(E,M)<0; (ii) C(E,M)=0, y'(0)<0;  (iii) C(E,M)>0, y'(0)^2\geq 24(p-1)C(E,M)y(0). $$
In fact, in the above conditions, we have $T<+\infty$ and
$$\lim_{t\to T}y(t)=0$$
this together with
$$\|u_0\|_{L_x^2}^2=\|u(t)\|_{L^2}^2\leq\big\||x|u(t)\big\|_{L^2}\|u(t)\|_{\dot H^1}.$$
implies
\begin{equation}\label{equ:blup}
\lim_{t\to T}\|u(t)\|_{\dot{H}^1}=+\infty.
\end{equation}

{\bf Case 2: $u_0\in H^1_{\rm rad}(\R^3)$.}
Let $\phi$ be a smooth, radial function satisfying $|\pa^2_r\phi(r)|\leq 2$, $\phi(r)=r^2 $ for $r\leq 1$, and $\phi(r)=0$ for $r\geq3$.
For $R\geq1 $, we define $$\phi_R(x)=R^2\phi\big(\tfrac{|x|}{R}\big)\;\text{ and }\; V_R(x)=\int_{\R^3} \phi_R(x)|u(t,x)|^2dx.$$
Let $u(t,x)$ be a radial solution to \eqref{equ1.1}, then by a direct computation,  we have by \eqref{equ:virit}
\begin{equation}\label{equ:vr}
  \partial_t V_R(x)=2 {\rm Im} \int_{\R^3}[\overline{u}\pa_j u](t,x)\pa_j\phi_R(x)] dx,
\end{equation}
 and
\begin{align*}
\pa^2_t V_R(x)=&4{\rm Re}\int_{\R^3} \pa_ju\pa_k\bar{u}\pa_j\pa_k\phi_R-\int_{\R^3}|u|^2\Delta^2\phi_R\;dx\\\nonumber
&-K\int_{\R^3}\frac{|u|^2}{|x|^2} \phi_R'\;dx
-\frac{2(p-1)}{p+1}\int_{\R^3}|u|^{p+1}\Delta \phi_R\;dx\\
= & 4\int_{\R^3} \phi_R''|\nabla u|^2dx-K\int_{\R^3}\frac{|u|^2}{|x|^2} \phi_R'\;dx-\int_{\R^3}\left[
\Delta^2\phi_R |u(t,x)|^2
 +\frac{2(p-1)}{p+1}\Delta\phi_R(x)|u|^{p+1}(t,x)\right]
 dx\\
 =&8\int_{\R^3}|\nabla u|^2\;dx-2K\int_{\R^3}\frac{|u|^2}{|x|}\;dx-\frac{12(p-1)}{p+1}\int_{\R^3}|u|^{p+1}\;dx-\int_{\R^3}
\Delta^2\phi_R |u(t,x)|^2\;dx\\
 &-4\int_{\R^3}|\nabla u|^2(2-\phi_R'')\;dx+K\int_{\R^3}\frac{|u|^2}{|x|}(2-\phi_R')\;dx+\frac{2(p-1)}{p+1}\int_{\R^3}|u|^{p+1}(6-\Delta\phi_R)\;dx\\
 \leq&12(p-1)E(u)-2(3p-7)\|\nabla u\|_{L^2}^2+6K\int_{\R^3}\frac{|u|^2}{|x|}\;dx\\
 &-4\int_{\R^3}|\nabla u|^2(2-\phi_R'')\;dx+C\int_{|x|\geq R}\big(\tfrac{|u|^2}{R}+|u|^{p+1}\big)\;dx.
\end{align*}
By the radial Sobolev inequality, we have
\begin{align*}
\|f\|_{L^\infty(|x|\geq
R)}\leq&\frac{c}{R}\|f\|_{L^2_x(|x|\geq R)}^\frac12\|\nabla
f\|_{L^2_x(|x|\geq R)}^\frac12.
\end{align*}
Therefore, by mass conservation and Young's inequality, we know that
for any $\epsilon>0$ there exist sufficiently large $R$ such that for $K\leq0$
\begin{align*}
\pa_t^2V(t)\leq&12(p-1)E(u)-2(3p-7-\epsilon)\|u\|_{\dot{H}^1}^2
+ \epsilon^2\\
\leq&12(p-1)E(u_0)+ \epsilon^2,
\end{align*}
and for $K>0$ by using \eqref{equ:kpos} with $C_{p,K}=\frac{3p-7-\delta}{3K}$ and $0<\delta\ll1$
\begin{align*}
\pa_t^2V(t)\leq&12(p-1)E(u)-(\delta-\epsilon)\|u\|_{\dot{H}^1}^2+\frac{18K^2}{3p-7-\delta}M(u)
+ \epsilon^2\\
\leq&12(p-1)E(u_0)+\frac{18K^2}{3p-7-\delta}M(u_0)+ \epsilon^2,
\end{align*}
for any $3p-7>\delta>\epsilon>0$.

 Finally, if we choose $\epsilon$
sufficient small, we can obtain
\begin{equation}\label{equ:vrt}
  \pa_t^2V_R(t)\leq \begin{cases}
                      6(p-1)E(u_0), & \mbox{if } K\leq0 \\
                      6(p-1)E(u_0)+\frac{9K^2}{3p-7-\delta}M(u_0), & \mbox{if } K>0,
                    \end{cases}
\end{equation}
which implies that $u$ blows up in finite time by the same argument as Case 1, since for the case $K>0$, the assumption
$$E(u_0)+\frac{3K^2}{2(3p-7)(p-1)}M(u_0)<0,$$
shows that there exists $0<\delta\ll1$ such that
$$6(p-1)E(u_0)+\frac{9K^2}{3p-7-\delta}M(u_0)<0.$$


\begin{center}

\end{center}

\end{document}